\documentclass[11pt, a4paper,leqno]{amsart}
\usepackage{amsmath,amsthm,amscd,amssymb,amsfonts, amsbsy}
\usepackage{latexsym}
\usepackage{exscale}

\usepackage{pgf}

\DeclareMathOperator*{\argmin}{arg\min}

\renewcommand{\L}{\operatorname{L}} 
\newcommand{\mcl}{\mathcal}

\def\J{\mathcal J}

\newcommand{\g}{\mathbf{g}}
\renewcommand{\j}{\mathbf{j}}

\newcommand{\Ll}{\left}
\newcommand{\Rr}{\right}

\calclayout
\allowdisplaybreaks

\newtheorem{theorem}{Theorem}[section]

\newtheorem{lemma}{Lemma}[section]

\theoremstyle{definition}
\newtheorem{definition}{Definition}
\newtheorem{remark}{Remark}[section]

\numberwithin{equation}{section}

\newcommand{\beq}{\begin{equation}}
\newcommand{\bea}[1]{\begin{array}{#1} }
\newcommand{\eeq}{ \end{equation}}
\newcommand{\ea}{ \end{array}}

\def \m {{\mu}}

\def \d {{\delta}}

\def\Xint#1{\mathchoice
{\XXint\displaystyle\textstyle{#1}}%
{\XXint\textstyle\scriptstyle{#1}}%
{\XXint\scriptstyle\scriptscriptstyle{#1}}%
{\XXint\scriptscriptstyle%
\scriptscriptstyle{#1}}%
\!\int}
\def\XXint#1#2#3{{\setbox0=\hbox{$#1{#2#3}{%
\int}$ }
\vcenter{\hbox{$#2#3$ }}\kern-.6\wd0}}
\def\barint{\,\Xint -} 
\def\bariint{\barint_{} \kern-.4em \barint}
\def\bariiint{\bariint_{} \kern-.4em \barint}

\def\mean#1{\mathchoice%
          {\mathop{\kern 0.2em\vrule width 0.6em height 0.69678ex depth -0.58065ex
                  \kern -0.8em \intop}\nolimits_{\kern -0.4em#1}}%
          {\mathop{\kern 0.1em\vrule width 0.5em height 0.69678ex depth -0.60387ex
                  \kern -0.6em \intop}\nolimits_{#1}}%
          {\mathop{\kern 0.1em\vrule width 0.5em height 0.69678ex
              depth -0.60387ex
                  \kern -0.6em \intop}\nolimits_{#1}}%
          {\mathop{\kern 0.1em\vrule width 0.5em height 0.69678ex depth -0.60387ex
                  \kern -0.6em \intop}\nolimits_{#1}}}

\def\vintslides_#1{\mathchoice%
          {\mathop{\kern 0.1em\vrule width 0.5em height 0.697ex depth -0.581ex
                  \kern -0.6em \intop}\nolimits_{\kern -0.4em#1}}%
          {\mathop{\kern 0.1em\vrule width 0.3em height 0.697ex depth -0.604ex
                  \kern -0.4em \intop}\nolimits_{#1}}%
          {\mathop{\kern 0.1em\vrule width 0.3em height 0.697ex depth -0.604ex
                  \kern -0.4em \intop}\nolimits_{#1}}%
          {\mathop{\kern 0.1em\vrule width 0.3em height 0.697ex depth -0.604ex
                  \kern -0.4em \intop}\nolimits_{#1}}}

\newcommand{\aveint}[2]{\mathchoice%
          {\mathop{\kern 0.2em\vrule width 0.6em height 0.69678ex depth -0.58065ex
                  \kern -0.8em \intop}\nolimits_{\kern -0.45em#1}^{#2}}%
          {\mathop{\kern 0.1em\vrule width 0.5em height 0.69678ex depth -0.60387ex
                  \kern -0.6em \intop}\nolimits_{#1}^{#2}}%
          {\mathop{\kern 0.1em\vrule width 0.5em height 0.69678ex depth -0.60387ex
                  \kern -0.6em \intop}\nolimits_{#1}^{#2}}%
          {\mathop{\kern 0.1em\vrule width 0.5em height 0.69678ex depth -0.60387ex
                  \kern -0.6em \intop}\nolimits_{#1}^{#2}}}

\def\eqn#1$$#2$${\begin{equation}\label#1#2\end{equation}}
\def\charfn_#1{{\raise1.2pt\hbox{$\chi
_{\kern-1pt\lower3pt\hbox{{$\scriptstyle#1$}}}$}}}

\def\qq1{q_*}
\def\q2{q_{**}}

\newdimen\vintbar
\def\vint{-\kern-\vintbar\int}

\def\J{\mathcal J}
\def\K{\mathcal K}

\def\L{\mathcal L}

\def\0{\boldsymbol 0}

\newcommand{\R}{\mathbb R}

\newcommand{\N}{\mathbb N}

\newtoks\by
\newtoks\paper
\newtoks\book
\newtoks\jour
\newtoks\yr
\newtoks\pages
\newtoks\vol
\newtoks\publ

\def\name[#1, #2]{#1 #2}
\def\ota{{\hbox{\bf ???}}}
\def\cLear{\by=\ota\paper=\ota\book=\ota\jour=\ota\yr=\ota
\pages=\ota\vol=\ota\publ=\ota}
\def\endpaper{\the\by, \textit{\the\paper},
{\the\jour} \textbf{\the\vol} (\the\yr), \the\pages.\cLear}
\def\endbook{\the\by, \textit{\the\book},
\the\publ, \the\yr.\cLear}
\def\endpap{\the\by, \textit{\the\paper}, \the\jour.\cLear}
\def\endproc{\the\by, \textit{\the\paper}, \the\book, \the\publ,
\the\yr, \the\pages.\cLear}

\renewcommand{\d}{\, \mathrm{d}} 

\begin{document}
\title[Nonlinear Kolmogorov-Fokker-Planck equations]{On regularity and existence of weak solutions to nonlinear Kolmogorov-Fokker-Planck type\\ equations with rough coefficients}

\address{Prashanta Garain \\Department of Mathematics, Uppsala University\\
S-751 06 Uppsala, Sweden}
\email{prashanta.garain@math.uu.se,\,\, pgarain92@gmail.com}

\address{Kaj Nystr\"{o}m\\Department of Mathematics, Uppsala University\\
S-751 06 Uppsala, Sweden}
\email{kaj.nystrom@math.uu.se}


\author{Prashanta Garain and Kaj Nystr{\"o}m}
\maketitle
\begin{abstract}
\noindent \medskip
We consider nonlinear Kolmogorov-Fokker-Planck type equations of the form
\begin{equation}\label{abeqn}
(\partial_t+X\cdot\nabla_Y)u=\nabla_X\cdot(A(\nabla_X u,X,Y,t)).
\end{equation}
 The function $A=A(\xi,X,Y,t):\R^m\times\R^m\times\R^m\times\R\to\R^m$ is assumed to be  continuous with respect to $\xi$, and measurable with respect to $X,Y$ and $t$. $A=A(\xi,X,Y,t)$ is allowed to be nonlinear but with linear growth. We establish higher integrability and local boundedness of weak sub-solutions, weak Harnack and Harnack
  inequalities, and H{\"o}lder continuity with quantitative estimates. In addition we establish existence and uniqueness of weak solutions to a Dirichlet problem in certain  bounded $X$, $Y$ and $t$ dependent domains. \\

\noindent
2020 {\em Mathematics Subject Classification.} 35K65, 35K70, 35H20, 35R03.
\noindent

\medskip

\noindent
{\it Keywords and phrases: Kolmogorov equation, parabolic, ultraparabolic, hypoelliptic,  nonlinear Kolmogorov-Fokker-Planck equations, existence, uniqueness, regularity.}
\end{abstract}

    \setcounter{equation}{0} \setcounter{theorem}{0}

   \section{Introduction and statement of main results}
Several important evolution equations arising in kinetic theory, mathematical physics and probability can be written in the form
\begin{eqnarray}\label{e-kolm-ndcol}
 (\partial_t+X\cdot\nabla_Y)f=\mathcal{Q}(f,\nabla_X f,X,Y,t),
    \end{eqnarray}
where $(X,Y,t):=(x_1,...,x_{m},y_1,...,y_{m},t)\in \mathbb R^{m}\times\mathbb R^{m}\times\mathbb R=\mathbb R^{N+1}$, $N=2m$, $m\geq 1$,  and the coordinates $X=(x_1,...,x_m)$ and $Y=(y_{1},...,y_{m})$ are,
respectively, the velocity and the position of the system.  In its simplest form,
$$\mathcal{Q}(f,\nabla_X f,X,Y,t)=\nabla_X\cdot\nabla_Xf=\Delta_Xf,$$
 the equation in \eqref{e-kolm-ndcol} was introduced and studied by Kolmogorov in a famous note published in 1934 in Annals of Mathematics, see \cite{K}. In this case Kolmogorov noted that the equation in \eqref{e-kolm-ndcol} is an example of a degenerate parabolic
operator having strong regularity properties and he proved that the equation has a fundamental solution which is smooth off its diagonal.  In fact, in this case the equation
in \eqref{e-kolm-ndcol} is hypoelliptic, see \cite{Hormander}.

In kinetic theory,  $f$ represents the evolution of a particle
distribution
$$f(X,Y,t):U_X\times U_Y\times \mathbb R_+\to\mathbb R,\quad U_X,\ U_Y\subset \mathbb R^m,$$
subject to geometric restrictions and models for the interactions and collisions between particles.  In this case the left-hand side in \eqref{e-kolm-ndcol} describes the evolution of $f$ under the action of transport, with the free
streaming operator. The right-hand side describes elastic collisions through the nonlinear Boltzmann collision operator.  The Boltzmann equation is an integro-
(partial)-differential equation with nonlocal operator in the kinetic variable $X$.  The
Boltzmann equation is a fundamental equation in kinetic theory in the sense that it has been derived rigorously, at least in some settings, from microscopic first
principles. In the case of so called Coulomb interactions the Boltzmann collision operator is
ill-defined and Landau proposed an alternative operator for these  interactions, this operator is now called the Landau or the Landau-Coulomb operator. This operator can be stated as in \eqref{e-kolm-ndcol} with
\begin{eqnarray}\label{e-kolm-ndc1}
 \mathcal{Q}(f,\nabla_X f,X,Y,t)=\nabla_X\cdot(A(f)\nabla_Xf+B(f)f),
    \end{eqnarray}
    where again $A(f)=A(f)(X,Y,t)$ and $B(f)=B(f)(X,Y,t)$ are  nonlocal operators in the variable $X$. In this case the equation in \eqref{e-kolm-ndcol} is a nonlinear, or rather quasilinear, drift-diffusion equation with coefficients given by convolution like averages of the unknown.  As mentioned above  the Landau equation is  considered fundamental because of its close link to the Boltzmann equation for Coulomb interactions.

In the case of long-range interactions, the Boltzmann and Landau-Coulomb operators show local ellipticity provided
the solution enjoys some pointwise bounds on the associated hydrodynamic fields and the local entropy. Indeed, assuming certain uniform in $(Y, t)\in U_Y\times I$ bounds on local mass, energy, and entropy, see \cite{LN1,Mou},  one can prove that
           \begin{eqnarray*}\label{e-kolm-nd22}
           0<\Lambda^{-1} I\leq A(f)(X,Y,t)\leq \Lambda I,\quad |B(f)(X,Y,t)|\leq \Lambda,
       \end{eqnarray*}
       for some constant $\Lambda\geq 1$ and for $(X,Y,t)\in U_X\times U_Y\times I$, i.e., under these assumptions the operator $\mathcal{Q}$ in \eqref{e-kolm-ndc1} and in the Landau equation becomes locally uniformly elliptic. As a consequence, and as global well posedness for the Boltzmann
equation and the construction of solutions in the large is an outstanding open problem, the study of conditional regularity for the Boltzmann and Landau equations has become a way to make progress on the regularity issues for these equations. We refer to \cite{MR715658,d15, MR2765747,dvI, PLLCam,Mou, MR554086,review} for more on the connections between Kolmogorov-Fokker-Planck equations, the Boltzmann and Landau equation, statistical physics and conditional regularity.

Based on the idea of conditional regularity one is lead to study the local regularity of weak solutions to the equation in
\eqref{e-kolm-ndcol} with
    \begin{eqnarray}\label{e-kolm-ndflbon-ll}
              \mathcal{Q}(f,\nabla_X f,X,Y,t)=\nabla_X\cdot(A(X,Y,t)\nabla_Xf)+B(X,Y,t)\nabla_Xf,
    \end{eqnarray}
    assuming  that $A$ is measurable, bounded and uniformly elliptic, and that $B$ is bounded. In \cite{GIMV}, see also \cite{Gu, GI, JC21} for subsequent developments, the authors extended,  for equations as in \eqref{e-kolm-ndcol} assuming \eqref{e-kolm-ndflbon-ll}, the  De Giorgi-Nash-Moser (DGNM) theory, which in its original form only considers elliptic
or parabolic equations in divergence form, to hypoelliptic equations with rough coefficients including the one in \eqref{e-kolm-ndcol} assuming \eqref{e-kolm-ndflbon-ll}. \cite{GIMV} has spurred considerable activity in the field, see below for a literature review, as the results proved give the correct scale- and translation-invariant estimates for local H{\"o}lder continuity and the Harnack inequality for weak solutions.

In this paper we consider equations as in \eqref{e-kolm-ndcol} with
              \begin{eqnarray}\label{e-kolm-ndflbon}
              \mathcal{Q}(f,\nabla_X f,X,Y,t)=\nabla_X\cdot(A(\nabla_X f,X,Y,t)),
    \end{eqnarray}
    subject to conditions on $A$ which allow $A$ to be a nonlinear function of $\nabla_X f$. In this case we refer to the equations in  \eqref{e-kolm-ndcol} as nonlinear Kolmogorov-Fokker-Planck type equations with rough coefficients. Our contributions is twofold. First, we establish higher integrability (Theorem \ref{nthm1}) and local boundedness (Theorem \ref{locbd}) of weak sub-solutions, weak Harnack and Harnack
  inequalities (Theorem \ref{thm3}), and H{\"o}lder continuity with quantitative estimates (Theorem \ref{thm4}), for the equation
    \begin{equation}\label{maineqn}
(\partial_t+X\cdot\nabla_Y)u=\nabla_X\cdot(A(\nabla_X u,X,Y,t)).
\end{equation}
Second, we establish existence and uniqueness, in certain bounded $X$, $Y$ and $t$ dependent domains, for a Dirichlet problem
involving the equation in \eqref{maineqn}
also allowing for boundary data and a right hand side (Theorem \ref{weakdp1}).  In the linear case, if $A(X,Y,t)$ is a uniformly elliptic positive definite matrix with bounded measurable coefficients, then $A(\xi,X,Y,t)=A(X,Y,t)\xi$ satisfies the hypothesis we impose on the symbol $A$,  and in this case the equation in \eqref{maineqn} reduces to  the equation
\begin{equation}\label{pro2}
(\partial_t+X\cdot\nabla_Y)u=\nabla_X\cdot(A(X,Y,t)\nabla_X u).
\end{equation}
Concerning regularity, our results therefore generalize \cite{GIMV}, \cite{GI,JC21}, to nonlinear Kolmogorov-Fokker-Planck type equations with rough coefficients.

To the best of our knowledge, nonlinear equations of the form in \eqref{maineqn} have so far not been investigated in the literature, and the purpose of this paper is to contribute to the regularity and existence theory for these equations. We believe that generalizations of the  De Giorgi-Nash-Moser (DGNM) theory to nonlinear Kolmogorov-Fokker-Planck type equations with rough coefficients are relevant and interesting. We also believe that our treatment of the Dirichlet problem is new and enlightening.

 \subsection{The symbol $A$} We consider equations as in \eqref{maineqn} subject to conditions on $A$. Concerning the symbol $A$ our baseline assumption is that $A$ belongs to the class $M(\Lambda)$, where $\Lambda\in [1,\infty)$ is a constant. In our treatment of the Dirichlet problem we will need to impose stronger conditions on $A$ and we will assume that $A$ belongs to the class $R(\Lambda)$. In the following $\cdot$ denotes the standard Euclidean scalar product in $\mathbb R^m$.
 \begin{definition}\label{class1} Let $\Lambda\in [1,\infty)$. Then $A$ is said to belong to the class $M(\Lambda)$ if
 $A=A(\xi,X,Y,t):\R^m\times\R^m\times\R^m\times\R\to\R^m$ is continuous with respect to $\xi$,  measurable with respect to $X,Y$ and $t$, and
 \begin{align}\label{assump2}
(i)&\quad |A(\xi,X,Y,t)|\leq \Lambda|\xi|,\notag\\
(ii)&\quad A(\xi,X,Y,t)\cdot \xi\geq \Lambda^{-1}|\xi|^2,\notag\\
(iii)&\quad A(\lambda\xi,X,Y,t)=\lambda A(\xi,X,Y,t) \quad \forall \lambda\in\R\setminus\{0\},
\end{align}
for almost every $(X,Y,t)\in\R^{N+1}$ and for all $\xi\in\mathbb{R}^m$.
 \end{definition}

  \begin{definition}\label{class2} Let $\Lambda\in [1,\infty)$. Then $A$ is said to belong to the class $R(\Lambda)$ if
 $A=A(\xi,X,Y,t):\R^m\times\R^m\times\R^m\times\R\to\R^m$ is continuous with respect to $\xi$,  measurable with respect to $X,Y$ and $t$, and
 \begin{align}\label{assump1}
(i)&\quad |A(\xi_1,X,Y,t)-A(\xi_2,X,Y,t)|\leq \Lambda|\xi_1-\xi_2|,\notag\\
(ii)&\quad (A(\xi_1,X,Y,t)-A(\xi_2,X,Y,t))\cdot(\xi_1-\xi_2)\geq \Lambda^{-1}|\xi_1-\xi_2|^2,\notag\\
(iii)&\quad A(\lambda\xi,X,Y,t)=\lambda A(\xi,X,Y,t) \quad \forall \lambda\in\R\setminus\{0\},
\end{align}
for almost every $(X,Y,t)\in\R^{N+1}$ and for all $\xi_1,\,\xi_2,\,\xi\in\mathbb{R}^m$.
 \end{definition}

 \begin{remark}
 Note that \eqref{assump1}-$(iii)$  implies that $A(0,X,Y,t)=0$ for a.e. $(X,Y,t)\in\R^{N+1}$. Hence  we deduce from \eqref{assump1}-$(i),\,(ii)$ and $(iii)$ that
$R(\Lambda)\subset M(\Lambda)$.
\end{remark}

\subsection{Dilations and group law}
 We will often use the notation $(Z,t)=(X,Y,t)\in \R^{N+1}$ to denote points.  The natural family of dilations for our operators and equations, $(\delta_r)_{r>0}$, on $\R^{N+1}$,
is defined by
\begin{equation}\label{dil.alpha.i}
 \delta_r (X,Y,t) =(r X, r^3 Y,r^2 t),
\end{equation}
for $(X,Y,t) \in \R^{N +1}$,  $r>0$.  Our classes of operators  are closed under the  group law
\begin{equation}\label{e70}
 (\tilde Z,\tilde t)\circ (Z,t)=(\tilde X,\tilde Y,\tilde t)\circ (X, Y,t)=(\tilde X+X,\tilde Y+Y+t\tilde X,\tilde t+t),
\end{equation}
where $(Z,t),\ (\tilde Z,\tilde t)\in \R^{N+1}$. Note that
\begin{equation}\label{e70+}
(Z,t)^{-1}=(X,Y,t)^{-1}=(-X,-Y+tX,-t),
\end{equation}
and hence
\begin{equation}\label{e70++}
 (\tilde Z,\tilde t)^{-1}\circ (Z,t)=(\tilde  X,\tilde  Y,\tilde  t)^{-1}\circ (X,Y,t)=(X-\tilde  X,Y-\tilde  Y-(
t-\tilde  t)\tilde  X,t-\tilde  t),
\end{equation}
whenever $(Z,t),\ (\tilde Z,\tilde t)\in \R^{N+1}$. Given $(Z,t)=(X,Y,t)\in \R^{N+1}$ we let \begin{equation}\label{kolnormint}
  \|(Z, t)\|= \|(X,Y, t)\|:=|(X,Y)|\!+|t|^{\frac{1}{2}},\quad |(X,Y)|=\big|X\big|+\big|Y\big|^{1/3}.
\end{equation}
Given  ${r}>0$ and  $(\tilde Z,\tilde t)=(\tilde X,\tilde Y,\tilde t)\in\R^{N+1}$,  we let
\begin{align}\label{cyl1}
Q_{r}:=\{(X,Y,t):|X|<{r},|Y|<{r}^3,-{r}^2<t<0\},\quad Q_{r}(\tilde Z,\tilde t):=(\tilde Z,\tilde t)\circ Q_{r}.
\end{align}
We refer to $Q_{r}(\tilde Z,\tilde t)$ as a cylinder centered at $(\tilde Z,\tilde t)$ and of radius ${r}$.

\subsection{Statement of main results: regularity of weak solutions}   We here state the regularity part of our results, Theorem \ref{nthm1}-Theorem \ref{thm4}. These theorem
are derived under the assumption that the symbol $A$ belongs to the class $M(\Lambda)$ introduced in Definition \ref{class1}. For the notions of weak sub-solutions, super-solutions and solutions, we refer to Definition \ref{wksoldef} below. For the definitions of function spaces used we refer to the bulk of the paper.

\begin{theorem}[Higher integrability]\label{nthm1}
Let $(Z_0,t_0)=(X_0,Y_0,t_0)\in\R^{N+1},\,0<r_1<r_0\leq 1$ and let $u$ be a non-negative weak sub-solution  to \eqref{maineqn} in an open set of $\R^{N+1}$ containing $ Q_{r_0}(Z_0,t_0)$ in the sense of Definition \ref{wksoldef} below.
Then for any $q\in[2,2+{1}/{m})$ and $s\in[0, {1}/{3})$, we have\footnote{\,\,$W_Y^{s,1}$ denotes the fractional Sobolev space.}
\begin{equation}\label{hint2}
\|u\|_{L^q(Q_{r_1}(Z_0,t_0))}\leq c_1\Big(2+\frac{1}{m}-q\Big)^{-1}\|u\|_{L^2(Q_{r_0}(Z_0,t_0))},
\end{equation}

\begin{equation}\label{ngest}
\|u\|_{L_{t,X}^1 W_Y^{s,1}(Q_{r_1}(Z_0,t_0))}\leq c_2\Big(\frac{1}{3}-s\Big)^{-1}\|u\|_{L^2(Q_{r_0}(Z_0,t_0))}.
\end{equation}
Here
$$
c_1=\Big(1+\frac{1}{r_0-r_1}\Big)c,\quad c_2=r_0^{1+2m}\Big(1+\frac{1}{r_0-r_1}\Big)c,
$$
where
$$
c=c(m,\Lambda)\Big(1+\frac{1}{(r_0-r_1)^2}+\frac{|X_0|+r_0}{(r_0-r_1)r_1^{2}}+\frac{1}{(r_0-r_1)r_1}\Big),
$$
for some constant $c(m,\Lambda)\geq 1$.
\end{theorem}

\begin{theorem}[Local boundedness]\label{locbd}
Let $(Z_0,t_0)=(X_0,Y_0,t_0)\in\R^{N+1},\,0<r_{\infty}<r_0\leq 1$ and let $u$ be  a non-negative weak sub-solution  to \eqref{maineqn} in an open set of $\R^{N+1}$ containing $ Q_{r_0}(Z_0,t_0)$ in the sense of Definition \ref{wksoldef} below. Then for any $p>0$, there exists a constant $c=c(m,\Lambda)\geq 1$ and $\theta=\theta(m)>1$ such that
\begin{equation}\label{lest}
\sup_{Q_{r_\infty}(Z_0,t_0)}u\leq c\Big(\frac{1+|X_0|}{r_{\infty}^2 (r_0-r_{\infty})^3}\Big)^\frac{\theta}{p}\|u\|_{L^p(Q_{r_0}(Z_0,t_0))}.
\end{equation}
\end{theorem}

\begin{theorem}[Harnack inequalities]\label{thm3} Let $u$ be a non-negative weak super-solution to \eqref{maineqn} in an open set of $\R^{N+1}$ containing $ Q_1$ in the sense of Definition \ref{wksoldef} below. Then there exists $\zeta>0$ and $c\geq 1$, both depending only on
$m$ and $\Lambda$ such that
  \begin{equation}
    \label{eq:w-Harnack-stat}
    \left( \iiint_{\tilde Q_{{r_0}/{2}} ^-} u^\zeta (X,Y,t) \d X\d Y\d t
    \right)^{{1}/{\zeta}}
    \leq c\inf_{ Q_{{r_0}/{2}}}u,
  \end{equation}
  where $r_0={1}/{20}$ and
  $\tilde Q_{{r_0}/{2}}^{-}:=Q_{{r_0}/{2}}(0,0,-{19}r_0^2/8)$. Furthermore, if $u$ is  a non-negative weak solution  to \eqref{maineqn} in an open set of $\R^{N+1}$ containing $ Q_1$, then
  \begin{equation}
    \label{eq:s-Harnack-stat}
    \sup_{\tilde Q_{{r_0}/{4}} ^{-}} u \leq c \inf_{Q_{{r_0}/{4}}} u,
  \end{equation}
  where
  $\tilde Q_{{r_0}/{4}}^{-}:=Q_{{r_0}/{4}}(0,0,-{19}r_0^2/8)$.
  \end{theorem}

\begin{theorem}[H\"older continuity]\label{thm4} Let  $u$ be a weak solution  to \eqref{maineqn} in an open set of $\R^{N+1}$ containing $ Q_2$ in the sense of Definition \ref{wksoldef} below. Then there exists $\alpha\in(0,1)$ and $c\geq 1$, both depending only on $m$ and $\Lambda$ such that
\begin{equation}\label{hest}
\frac{|u(X_1,Y_1,t_1))-u(X_2,Y_2,t_2)|}{\|(X_2,Y_2,t_2)^{-1}\circ(X_1,Y_1,t_1)\|^\alpha}\leq c\|u\|_{L^2(Q_2)},
\end{equation}
whenever $(X_1,Y_1,t_1),(X_2,Y_2,t_2)\in Q_1,\,(X_1,Y_1,t_1)\neq (X_2,Y_2,t_2)$.
\end{theorem}

\subsection{Statement of main results: existence and uniqueness for a Dirichlet problem} We here state the existence and uniqueness part of our results, Theorem \ref{weakdp1}. Throughout the paper we let $U_X\subset\mathbb R^m$ be a bounded Lipschitz domain and let $V_{Y,t}\subset \mathbb R^{m}\times \mathbb R$ be a bounded domain with boundary which is $C^{1,1}$-smooth, {i.e. $C^{1}$ with respect to $Y$ as well as $t$.} Let $N_{Y,t}$ denote the outer unit normal to $V_{Y,t}$. We establish existence and uniqueness of  weak solutions
to a formulation of the Dirichlet problem
  \begin{equation}\label{dpweak+}
\begin{cases}
	\nabla_X\cdot(A(\nabla_X u,X,Y,t))-(\partial_t+X\cdot\nabla_Y)u = g^*  &\text{in} \ U_X\times V_{Y,t}, \\
      u = g  & \text{on} \ \partial_\K(U_X\times V_{Y,t}).
\end{cases}
\end{equation}
 Here
    \begin{eqnarray}\label{Kolb}
\quad\partial_{\K}(U_X\times V_{Y,t}):=(\partial U_X\times V_{Y,t})\cup\{(X,Y,t)\in \overline{U_X}\times \partial V_{Y,t}\mid (X,1)\cdot N_{Y,t}<0\}.
    \end{eqnarray}
    $\partial_\K(U_X\times V_{Y,t})$ will be referred to as the Kolmogorov boundary of $U_X\times V_{Y,t}$, and the Kolmogorov boundary serves, in our context, as the natural substitute for the parabolic boundary used in the context of the Cauchy-Dirichlet problem for uniformly elliptic parabolic equations.
In particular, we study weak solutions in the sense of Definition \ref{weaks}. For the definition of the functional setting we refer to Section \ref{sec2}. We believe that the following result is of independent interest in particularly as we allow the symbol $A$ to depend nonlinearly on $\nabla_Xu$.
\begin{theorem}[Existence and uniqueness]\label{weakdp1}  Let $(g,g^*)\in W(U_X\times V_{Y,t})\times L_{Y,t}^2(V_{Y,t},{H}_X^{-1}(U_X))$ and assume that $A$ belongs to the class $R(\Lambda)$ introduced in Definition \ref{class2}. Then there exists a unique weak solution $u$ to the problem in
\eqref{dpweak+} in the sense of Definition \ref{weaks} below. Furthermore, there exists a constant $c$, depending only on $m$, $\Lambda$ and $U_X\times V_{Y,t}$, such that
    \begin{equation}\label{exest}
    \begin{split}
    ||u||_{W(U_X\times V_{Y,t})}&\leq c\bigl (||g||_{W(U_X\times V_{Y,t})}+||g^*||_{L_{Y,t}^2(V_{Y,t},{H}_X^{-1}(U_X))}\bigr ).
    \end{split}
    \end{equation}
    \end{theorem}

\subsection{Known regularity results} As mentioned, the equation in \eqref{pro2}, possibly also allowing for lower order terms, has attracted considerable attention in recent years. Anceschi-Cinti-Pascucci-Polidoro-Ragusa \cite{APR,CPP,PP} proved local boundedness of weak sub-solutions of \eqref{pro2} and some versions thereof. Their approach is based on the Moser's iteration technique, the use of  fundamental solutions and a Sobolev type inequality is crucial. It is worth noting that while the results in these papers are stated assuming only bounded and measurable coefficients, an implicit regularity assumption on the coefficients is imposed as the authors use a stronger notion of weak solutions assuming also
$(\partial_t+X\cdot\nabla_Y)u\in L^2_{\mathrm{loc}}$. It is unclear for what assumptions on the coefficients such weak solutions can be constructed. Bramanti-Cerutti-Manfredini-Polidoro-Ragusa \cite{BCM,MP,PRagusa} proved $L^p$ estimates, interior Sobolev regularity and local H\"older continuity of weak solutions of \eqref{pro2} imposing additional assumptions on the coefficients beyond  bounded, measurable and elliptic. In fact it was only recently that Golse-Imbert-Mouhot-Vasseur \cite{GIMV} proved local boundedness, Harnack inequality and local H\"older continuity of (true) weak solutions of \eqref{pro2} based on De-Giorgi and Moser's iteration technique. Still, it seems unclear to us how the authors actually resolve questions concerning the existence of weak solutions  unless   smooth coefficients are assumed qualitatively. However, subsequent developments have appeared in \cite{ Gu, GI, JC21}. A weak Harnack inequality for weak super-solutions of \eqref{pro2} has been obtained by Guerand-Imbert \cite{GI} and this has been generalized by Anceschi-Rebucci \cite{AR}. In \cite{JC21}, Guerand-Mouhot revisited the theory for the linear equation in \eqref{pro2}, also allowing for lower order terms, and gave  lucid, novel and short proofs of the De Giorgi
  intermediate-value Lemma, weak Harnack and Harnack
  inequalities, and the  H{\"o}lder continuity with quantitative
  estimates. \cite{JC21} is an essentially self-contained account of the linear theory.  Local H\"older continuity results are also proved in Wang-Zhang \cite{WZ2, WZ1, WZ3} for various linear analogues of \eqref{pro2}. We emphasize that all results mentioned concern linear equations. Zhu \cite{Zhu} proved local boundedness and local H\"older continuity of weak solutions of \eqref{pro2} when the drift term $\partial_t+X\cdot\nabla_Y$ is replaced by $\partial_t+b(X)\cdot \nabla_Y$ for some nonlinear function $b$.

\subsection{Known existence results} Boundary value problems for equations as in  \eqref{pro2} but in non-divergence form were studied by Manfredini \cite{Manfredini} who proved existence of strong solutions for the Dirichlet problem assuming H\"older continuous coefficients. Lanconelli-Lascialfari-Morbidelli \cite{LL, LM} considered a quasilinear case, still in non-divergence form, allowing the coefficients to depend not only on $(X,Y,t)$ but also the solution $u$,  and as a function of $(X,Y,t)$ the coefficients are assumed to be with H\"older continuous. In fact,  functional analytic approaches to weak solutions to  Kramers equation and Kolmogorov-Fokker-Planck equations have only recently been developed.
 Albritton-Armstrong-Mourrat-Novack \cite{AAMN} have developed a functional analytic approach to study well-posedness of Kramers equation, and its parabolic analogue
\begin{equation}\label{pKram}
\partial_t u-\Delta_X u+X\cdot\nabla_X u+X\cdot\nabla_Y u+b\cdot\nabla_X u=g^*,
\end{equation}
for suitable $g^*$. Equation \eqref{pKram} is often referred to as the kinetic Fokker-Planck equation. Litsg{\aa}rd-Nystr\"om \cite{LN} studied existence and uniqueness results for the (linear) Dirichlet problem associated with \eqref{pro2}, with rough coefficients $A$. In particular, in \cite{LN} Theorem \ref{weakdp1} is proved in the case when
$A(\xi,X,Y,t)=A(X,Y,t)\xi$. However, existence and uniqueness for \eqref{maineqn} do not seem to have been studied in the literature so far. It is important to note
that Theorem \ref{weakdp1} states, similar to \cite{LN}, the existence of a unique weak solution $u$ to the problem in
\eqref{dpweak+} in the sense of Definition \ref{weaks} below. The latter is, as it assumes no knowledge of underlying traces, trace spaces  and extension operators in the functional setting considered,  a weaker formulation of the Dirichlet problem compared to what one usually aims for. Indeed, this is one way to formulate a weak form of the
Dirichlet problem which circumvents a largely open problem in the context of kinetic Fokker-Planck equations, linear as well as non-linear, and that is the problem of a well defined trace operator and trace inequality. We refer to Section \ref{sec6} for more.

\subsection{Proofs}\label{proof}  The  regularity part of our results is modelled on the approach of  Golse-Imbert-Mouhot-Vasseur \cite{GIMV} and the work of Guerand-Mouhot \cite{JC21}. In fact, as can be seen from the very formulations of our regularity results, this part of our work is strongly influenced by \cite{JC21} and armed with Theorem \ref{nthm1} and Theorem \ref{locbd} we can to large extent refer to the corresponding arguments in \cite{JC21} for the proofs of Theorem \ref{thm3} and  Theorem \ref{thm4}. The new difficulties in our case  stem from the nonlinearity of $A$ in $\nabla_X u$. However, as we learn from the regularity theory for quasi-linear parabolic PDEs, see \cite{DGV} for example, a careful development of the
De Giorgi-Nash-Moser theory tends to be robust enough to handle the type of non-linearities considered in this paper. The higher integrability result in Theorem \ref{nthm1} is proved  by combining  the energy estimate in Lemma \ref{eng1} with a Sobolev regularity estimate and here it is important that $A$ has linear growth in $\nabla_X u$. In particular, in the proof of Theorem \ref{nthm1} one is lead, after  preliminaries and the use of an appropriate cut-off function, to conduct estimates for a (global) weak sub-solution $u_1$ to the equation
\begin{equation}\label{geqnint}
(\partial_t+X\cdot\nabla_Y)u_1\leq \nabla_X\cdot A(\nabla_X u_1, X,Y,t)+g^\ast,\ g^\ast:=-(\nabla_X\cdot F_1+F_0)\text{ in }\R^{N+1},
\end{equation}
where $F_1,F_0$ are in $L^2(\R^{N+1})$ and $u_1,F_1,F_0$ are supported in $Q_{r_0}(0,0,0)$.  To close the argument, as $u_1$ is only a weak sub-solution, it seems important to replace it by a function which actually solves an equation. In particular, to make this operational one needs to construct a weak solution $v$ to
\begin{equation}\label{geqnint+}
(\partial_t+X\cdot\nabla_Y)v=\nabla_X\cdot A(\nabla_X v, X,Y,t)+g^\ast,
\end{equation}
such that $v$ bounds $u_1$ from above. One approach to Sobolev regularity estimates is then attempt to use an approach based on Bouchut \cite{Bou} which implies a  Sobolev embedding
\begin{equation}\label{sobem}
H_{X,Y,t}^{{1}/{3}}(\R^{N+1})\to L_{X,Y,t}^q(\R^{N+1}),\quad q:=\frac{6(2m+1)}{6m+1}>2.
\end{equation}
 To get hold of the $H_{X,Y,t}^{{1}/{3}}(\R^{N+1})$ norm of $v$ one uses a result of Bouchut \cite{Bou} which gives control of $D_Y^{1/3}v,\ D_t^{1/3}v$ given energy estimates. To be able to bound
$u_1$ from above by $v$ as in \eqref{geqnint+} one seems to need  Theorem \ref{weakdp1} and the comparison principle that we prove in Theorem \ref{compa} below. As the result of Bouchut \cite{Bou} requires a solution which exists globally in time one can make this approach operational using Theorem \ref{weakdp1} to prove Theorem \ref{nthm1} with the cylinders in \eqref{cyl1} replaced by centered cylinders.
An alternative approach to Sobolev regularity estimates, which in the end gives Theorem \ref{nthm1} as stated, is to first observe
that if $u_1$ satisfies \eqref{geqnint}, then one deduces that the weak formulation of \eqref{geqnint} induces a positive distribution. One is therefore lead to prove estimates for $v$ satisfying
\begin{equation}\label{geqnint++}
(\partial_t+X\cdot\nabla_Y)v=\nabla_X\cdot A(\nabla_X v, X,Y,t)+g^\ast-\mu,
\end{equation}
where $\mu$ is now a positive measure. Due to the structure of $g^\ast$, Sobolev regularity estimates can then be deduced using a semi-classical approach via the fundamental solution associated to the linear equation $(\partial_t+X\cdot\nabla_Y)f=\Delta_X f$ originally studied by Kolmogorov, see Lemma 10 in \cite{JC21}. In the end, we follow this approach and here it is again important that $A$ has linear growth in $\nabla_X u$. Armed with the Sobolev regularity estimates the proofs of Theorem \ref{nthm1}-Theorem \ref{thm4} can be completed along the lines of the corresponding arguments in the linear case. Finally, to prove the existence and uniqueness result in Theorem \ref{weakdp1} we use a variational approach and proceed along the lines of \cite{AAMN,ABM,LN}. In particular, our argument is similar to the proof of Theorem 1.1 in \cite{LN}.

\subsection{Organization of the paper} In Section \ref{sec2} we introduce the functional setting and the notion of weak solutions. Section \ref{sec3} is devoted to a number of preliminary technical results to be used in the proofs of Theorem \ref{nthm1}-Theorem \ref{thm4}. Theorem \ref{nthm1}-Theorem \ref{thm4} are proved in
Section \ref{sec4}, and in the proof of Theorem \ref{thm3} and Theorem \ref{thm4} we for brevity  mainly refer to the corresponding arguments in \cite{JC21}. Theorem \ref{weakdp1} is proved in Section \ref{sec5}. In Section \ref{sec6} we mention a number of challenging problems for future research which we hope will inspire the community to look further into the topic of nonlinear Kolmogorov-Fokker-Planck type equations.

\section{The functional setting and weak solutions}\label{sec2}

\subsection{Function spaces} We denote by ${H}_X^1(U_X)$ the Sobolev
space of functions $g\in L_{}^2(U_X)$  whose distributional gradient in $U_X$ lies in $(L^2(U_X))^m$, i.e.
  \begin{eqnarray*}\label{fspace-}
{H}_X^1(U_X):=\{g\in L_{X}^2(U_X)\mid \nabla_Xg\in (L^2(U_X))^m\},
    \end{eqnarray*}
    and we set
    $$||g||_{{H}_X^1(U_X)}:=\bigl (||g||_{L^2(U_X)}^2+||\,|\nabla_Xg|\,||_{L^2(U_X)}^2\bigr )^{1/2},\ g\in {H}_X^1(U_X).$$
    We let ${H}_{X,0}^1(U_X)$ denote the closure of $C_0^\infty(U_X)$ in the norm of ${H}_X^1(U_X)$ and we recall, as $U_X$ is a bounded Lipschitz domain, that
    $C^\infty(\overline{U_X})$ is dense in ${H}_X^1(U_X)$. In particular, equivalently we could define ${H}_X^1(U_X)$ as the closure of $C^\infty(\overline{U_X})$
     in the norm $||\cdot||_{{H}_X^1(U_X)}$. {Note that as ${H}_{X,0}^1(U_X)$ is a Hilbert space it is reflexive, hence $({H}_{X,0}^1(U_X))^\ast=H_X^{-1}(U_X)$ and $(H_X^{-1}(U_X))^\ast={H}_{X,0}^1(U_X)$, where $()^\ast$ denotes the dual.
     Based on this we let ${H}_X^{-1}(U_X)$ denote the dual to ${H}_{X,0}^1(U_X)$ acting on functions in ${H}_{X,0}^1(U_X)$ through the
     duality pairing $\langle \cdot,\cdot\rangle:=\langle \cdot,\cdot\rangle_{H_X^{-1}(U_X),H_{X,0}^{1}(U_X)}$.} We let $L^2_{Y,t}(V_{Y,t},H_{X,0}^{1}(U_X))$ be the space of
     measurable function $u:V_{Y,t}\to H_{X,0}^{1}(U_X)$ equipped with the norm
     $$||u||^2_{L_{Y,t}^2(V_{Y,t},H_X^1(U_X))}:=\iint_{V_{Y,t}}||u(\cdot,Y,t)||_{{H}_X^1(U_X)}^2\, \d Y\d t.$$
     $L^2_{Y,t}(V_{Y,t},H_{X}^{-1}(U_X))$ is defined analogously. In analogy with the definition of ${H}_X^1(U_X)$, we let $W(U_X\times V_{Y,t})$ be the closure of $C^\infty(\overline{U_X\times V_{Y,t}})$ in the norm
     \begin{align}\label{weak1-+}
     ||u||_{W(U_X\times V_{Y,t})}&:=\bigl (||u||_{L_{Y,t}^2(V_{Y,t},H_X^1(U_X))}^2+||(\partial_t+X\cdot\nabla_Y)u||_{L_{Y,t}^2(V_{Y,t},{H}_X^{-1}(U_X))}^2\bigr )^{1/2}.
    \end{align}
    In particular, $W(U_X\times V_{Y,t})$ is a Banach space and  $u\in W(U_X\times V_{Y,t})$ if and only if \begin{eqnarray}\label{weak1-}
u\in L_{Y,t}^2(V_{Y,t},H_X^1(U_X))\quad\mbox{and}\quad (\partial_t+X\cdot\nabla_Y)u\in  L_{Y,t}^2(V_{Y,t},H_X^{-1}(U_X)).
    \end{eqnarray}
       Note that the dual of $L_{Y,t}^2(V_{Y,t},H_{X,0}^1(U_X))$, denoted by $(L_{Y,t}^2(V_{Y,t},H_{X,0}^1(U_X)))^\ast$, satisfies $$(L_{Y,t}^2(V_{Y,t},H_{X,0}^1(U_X)))^\ast=L_{Y,t}^2(V_{Y,t},H_X^{-1}(U_X)),$$
    and, as mentioned above,
    $$(L_{Y,t}^2(V_{Y,t},H_X^{-1}(U_X)))^\ast=L_{Y,t}^2(V_{Y,t},H_{X,0}^1(U_X)).$$
Finally, the spaces $L_{Y,t,\mathrm{loc}}^2(V_{Y,t},H_{X,\mathrm{loc}}^1(U_X))$, $L_{Y,t,\mathrm{loc}}^2(V_{Y,t},{H}_{X,\mathrm{loc}}^{-1}(U_X))$, and $W_{\mathrm{loc}}(U_X\times V_{Y,t})$ are defined in the natural way.  The topological boundary of $U_X\times V_{Y,t}$ is denoted by $\partial(U_X\times V_{Y,t})$. Let $N_{Y,t}$ denote the outer unit normal to $V_{Y,t}$. We define a subset
    $\partial_{\K}(U_X\times V_{Y,t})\subset\partial(U_X\times V_{Y,t})$, the Kolmogorov boundary of $U_X\times V_{Y,t}$, as in \eqref{Kolb}. We let $C^\infty_{\K,0}(\overline{U_X\times V_{Y,t}})$ and $C^\infty_{X,0}(\overline{U_X\times V_{Y,t}})$ be the set of functions in
$C^\infty(\overline{U_X\times V_{Y,t}})$ which vanish on $\partial_{\K}(U_X\times V_{Y,t})$ and $\{(X,Y,t)\in \partial{U_X}\times \overline{V_{Y,t}}\}$, respectively.  We let $W_0(U_X\times V_{Y,t})$ and  $W_{X,0}(U_X\times V_{Y,t})$ denote the closure in the norm of
$W(U_X\times V_{Y,t})$ of $C^\infty_{\K,0}(\overline{U_X\times V_{Y,t}})$ and $C^\infty_{X,0}(\overline{U_X\times V_{Y,t}})$, respectively.

\subsection{Weak solutions} We here introduce the notion of weak solutions.
\begin{definition}\label{wksoldef} Let $g^*\in L_{Y,t}^2(V_{Y,t},{H}_X^{-1}(U_X))$. A function $u\in W_{\mathrm{loc}}(U_X\times V_{Y,t})$  is said to be a  weak sub-solution (or super-solution) to the equation
\begin{align}\label{eeqa}
(\partial_t+X\cdot\nabla_Y)u-\nabla_X\cdot(A(\nabla_X u,X,Y,t))+g^*=0 \text{ in } \ U_X\times V_{Y,t},
\end{align}
if for every $V_X\times V_Y\times J\Subset U_X\times V_{Y,t}$,  and for all non-negative $\phi\in L_{Y,t}^2(V_Y\times J,H_{X,0}^1(V_X))$, we have
\begin{align}\label{wksol}
&\iiint_{V_X\times V_Y\times J}A(\nabla_X u, X,Y,t)\cdot\nabla_X\phi\,\d X \d Y \d t\notag\\
&+\iint_{V_Y\times J}\ \langle g^\ast(\cdot,Y,t)+ (\partial_t+X\cdot\nabla_Y)u(\cdot,Y,t),\phi(\cdot,Y,t)\rangle\, \d Y \d t\leq 0\quad (\text{ or }\geq ).
\end{align}
We say that $u\in W_{\mathrm{loc}}(U_X\times V_{Y,t})$ is a weak solution to the equation \eqref{eeqa} if equality holds in \eqref{wksol} without a sign restriction on $\phi$.
\end{definition}

Note that if $u$ is a weak sub-solution (or super-solution) of \eqref{eeqa} in the sense of Definition \ref{wksoldef} above, with $g^\ast\equiv 0$, then
\begin{align}\label{wksol1}
&\iint_{V_X\times V_Y}u(X,Y,t_2)\phi(X,Y,t_2)\,\d X \d Y-\iint_{V_X\times V_Y}u(X,Y,t_1)\phi(X,Y,t_1)\,\d X \d Y\notag\\
&-\int_{t_1}^{t_2}\iint u(\partial_t+X\cdot\nabla_Y)\phi\,\d X \d Y \d t+\int_{t_1}^{t_2}\iint A(\nabla_X u, X,Y,t)\cdot\nabla_X\phi\,\d X \d Y \d t\notag\\
&\leq 0\quad (\text{ or }\geq ),
\end{align}
whenever $\phi\in C^\infty((t_1,t_2), C^\infty_0(V_X\times V_Y))$, is non-negative function.
Furthermore, equality holds in \eqref{wksol1} for every weak solution $u$ of \eqref{eeqa} without a sign restriction on $\phi$.
\begin{remark}\label{wksolrmk} Assume  $g^\ast\equiv 0$. $(i)$ From Definition \ref{wksoldef}, it is clear that, if $u$ is a weak sub-solution (resp. super-solution or solution) of \eqref{eeqa} in $U_X\times V_{Y,t}$, then for any $k\in\R$, the function $v=(u-k)$ is also weak sub-solution (resp. super-solution or solution) of \eqref{eeqa} in $U_X\times V_{Y,t}$.\\
\noindent
$(ii)$ Using the homogeneity property $(iii)$ of $A$, it follows that, (a) for any $c\geq 0$, $cu$ is a weak sub-solution (resp. super-solution or solution) of \eqref{eeqa} in $U_X\times V_{Y,t}$, provided $u$ is a weak sub-solution (resp. super-solution or solution) of \eqref{eeqa} in $U_X\times V_{Y,t}$ and (b) $u$ is a weak solution of \eqref{eeqa} in $U_X\times V_{Y,t}$ if and only if $-u$ is a weak solution of \eqref{eeqa} in $U_X\times V_{Y,t}$.
\end{remark}

\subsection{The Dirichlet problem} Theorem \ref{weakdp1} is a statement concerning existence and uniqueness of  weak solutions
to a formulation of the Dirichlet problem in \eqref{dpweak+}. In particular, we study weak solutions in the following sense.
\begin{definition}\label{weaks} Consider $(g,g^\ast)\in W(U_X\times V_{Y,t})\times L_{Y,t}^2(V_{Y,t},{H}_X^{-1}(U_X))$. Given $(g,g^\ast)$, $u$ is said to be a weak solution to the problem in
\eqref{dpweak+} if
\begin{eqnarray}\label{weak1}
 u\in  W(U_X\times V_{Y,t}),\quad (u-g)\in  W_0(U_X\times V_{Y,t}),
    \end{eqnarray}
    and if
\begin{equation}\label{weak3}
    \begin{split}
     &\iiint_{U_X\times V_{Y,t}}\ A(\nabla_Xu,X,Y,t)\cdot \nabla_X\phi\, \d X \d Y \d t\\
    &+\iint_{V_{Y,t}}\ \langle g^*(\cdot,Y,t)+(\partial_t+X\cdot\nabla_Y)u(\cdot,Y,t),\phi(\cdot,Y,t)\rangle\, \d Y \d t=0,
\end{split}
\end{equation}
for all $ \phi\in L_{Y,t}^2(V_{Y,t},H_{X,0}^1(U_X))$ and where $\langle \cdot,\cdot\rangle=\langle \cdot,\cdot\rangle_{H_X^{-1}(U_X),H_{X,0}^{1}(U_X)}$ is the duality pairing in $H_X^{-1}(U_X)$. If in \eqref{weak3}, $=$ is replaced by $\leq (\geq)$ whenever $\phi\geq 0$, then $u$ is said to be a weak sub- (super-) solution of \eqref{dpweak+} respectively.
\end{definition}

\section{Technical lemmas} \label{sec3}
In this section we prove a number of technical results to be used in the proof of Theorem \ref{nthm1}-Theorem \ref{thm4}. Throughout the rest of the paper, we use the notation $s^{+}:=\max\{s,0\}$ for $s\in\R$. Moreover, from Section \ref{sec3}-\ref{sec4}, we assume that the symbol $A$ belongs to the class $M(\Lambda)$ introduced in Definition \ref{class1}.

\begin{lemma}\label{eng1}
Let $Z_0=(X_0,Y_0,t_0)\in\R^{N+1}$, $0<r_1<r_0$, be such that $Q_{r_0}(Z_0,t_0)\Subset U_X\times V_{Y,t}$. Let $u$ be a weak sub-solution of the equation \eqref{maineqn} in $U_X\times V_{Y,t}$  in the sense of Definition \ref{wksoldef}.
Then
\begin{align}\label{stform5uu}
&\sup_{t_0-r_{1}^2<t<t_0}\iint_{Q^t(Z_0,r_1)} u^2(X,Y,t)\,\d X \d Y+\Lambda^{-1}\iiint_{Q_{r_1}(Z_0,r_1)}|\nabla_X u|^2\,\d X \d Y \d t\notag\\
&\leq c c_{0,1}\iiint_{Q_{r_0}(Z_0,t_0)}u(X,Y,t)^2\,\d X \d Y \d t,
\end{align}
where $Q^t(Z_0,r):=\{(X,Y):(X,Y,t)\in Q_{r}(Z_0,t_0)\}$ for $r>0$,\,$c=c(m,\Lambda)\geq 1$ and
$$
c_{0,1}:=\frac{1}{(r_0-r_1)^2}+\frac{r_0+|X_0|}{(r_0 -r_1)r_1^{2}}+\frac{1}{(r_0 -r_1)r_1}+1.
$$
\end{lemma}
\begin{proof} Let  $t_1:=t_0-r_0^{2}$ and $t_2:=t_0$. Considering $l_1$, $l_2$, such that $t_1<l_1<l_2<t_2$, we introduce for $\epsilon>0$ the function $\theta_{\epsilon}\in W^{1,\infty}((t_1,t_2))$ by
\begin{equation}\label{auxfn}
\theta_{\epsilon}(t):=
\begin{cases}
0\text{ if }t_1\leq t\leq l_1-\epsilon,\\
1+\frac{t-l_1}{\epsilon},\text{ if }l_1-\epsilon<t\leq l_1,\\
1\text{ if }l_1<t\leq l_2,\\
1-\frac{t-l_2}{\epsilon}\text{ if }l_2\leq t\leq l_2+\epsilon,\\
0\text{ if }l_2+\epsilon<t\leq t_2.
\end{cases}
\end{equation}
Let $\psi\in[0,1]$ be smooth in $Q_{r_0}(Z_0,t_0)$ such that $\psi\equiv 1$ on $Q_{r_1}(Z_0,t_0)$ and $\psi\equiv 0$ outside $Q_{r_0}(Z_0,t_0)$ satisfying
$$|\nabla_X\psi|\leq\frac{c}{r_0- r_1},\quad |\nabla_Y\psi|\leq\frac{c}{(r_0-r_{1})r_{1}^2},\quad |\partial_t\psi|\leq\frac{c}{(r_0-r_{1})r_{1}},$$
for some constant $c=c(m)\geq 1$.

Consider the function $\phi(X,Y,t)=2u(X,Y,t)\psi^2(X,Y,t)\theta_{\epsilon}(t)$. We intend to test \eqref{wksol1} with $\phi$ and the following deductions are formal. However, as $u$ is a weak sub-solution of the equation \eqref{maineqn} in $U_X\times V_{Y,t}$  in the sense of Definition \ref{wksoldef}, we know that
$u\in W_{\mathrm{loc}}(U_X\times V_{Y,t})$ and as $W(U_X\times V_{Y,t})$ is defined as the closure of $C^\infty(\overline{{U_X\times V_{Y,t}}})$ in the norm introduced in \eqref{weak1-+} our deduction can be made rigorous a posteriori. Testing \eqref{wksol1} with $\phi(X,Y,t)$, letting $\epsilon\to 0$, and then adding
$$
\iiint u^{2}\partial_t (\psi^2)\,\d X \d Y \d t
$$
on both sides of the resulting inequality, we deduce that
\begin{align}\label{stform2}
&I(l_2)-I(l_1)+
2\iiint A(\nabla_X u,X,Y,t)\cdot\nabla_X(u\psi^2)\,\d X \d Y \d t\notag\\
&\leq \iiint u^{2}(\partial_t+X\cdot\nabla_Y) \psi^2\,\d X \d Y \d t,
\end{align}
where
\begin{align*}
I(t):=\iint\psi^2(X,Y,t)u^2(X,Y,t)\,\d X \d Y.
\end{align*}
Using \eqref{assump2}, \eqref{stform2} yields
\begin{align}\label{stform3}
&I(l_2)-I(l_1)+2\iiint\psi^2 A(\nabla_X u,X,Y,t)\cdot\nabla_X u\,\d X \d Y \d t\notag\\
&\leq \iiint u^{2}(\partial_t+X\cdot\nabla_Y) \psi^2\,\d X \d Y \d t\notag\\
&-4\iiint u\psi A(\nabla_X u,X,Y,t))\cdot\nabla_X\psi\,\d X \d Y \d t\notag\\
&\leq\iiint u^2\{(\partial_t+X\cdot\nabla_Y)\psi^2+4\Lambda^{3}(\psi+|\nabla_X\psi|)^2\}\,\d X \d Y \d t\notag\\
&+\iiint\Lambda^{-1}\psi^2|\nabla_X u|^2\,\d X \d Y \d t.
\end{align}
Furthermore, using \eqref{assump2}-$(i),(ii)$ we can continue the above estimate and conclude that
\begin{align}\label{stform4}
&I(l_2)-I(l_1)+\Lambda^{-1}\iiint_{Q^t(Z_0,r_0)\times (l_1,l_2)}\psi^2 |\nabla_X u|^2\,\d X \d Y \d t\notag\\
&\leq\iiint u^2\{(\partial_t+X\cdot\nabla_Y)\psi^2+4\Lambda^{3}(\psi+|\nabla_X\psi|)^2\}\,\d X \d Y \d t.
\end{align}
Using the properties of $\psi$ and first letting $l_1\to t_1$, and then letting $l_2\to t_2$ in \eqref{stform4}, we obtain
\begin{align}\label{stform5}
&\Lambda^{-1}\iiint_{Q_{r_1}(Z_0,t_0)}|\nabla_X u|^2\,\d X \d Y \d t\notag\\
&\leq\iiint_{Q_{r_0}(Z_0,t_0)} u^2\{(\partial_t+X\cdot\nabla_Y)\psi^2+4\Lambda^{3}(\psi+|\nabla_X\psi|)^2\}\,\d X \d Y \d t\notag\\
&\leq c c_{0,1}\iiint_{Q_{r_0}(Z_0,t_0)}u(X,Y,t)^2\,\d X \d Y \d t,
\end{align}
where $c=c(m,\Lambda)\geq 1$ and
$$
c_{0,1}:=\frac{1}{(r_0-r_1)^2}+\frac{r_0+|X_0|}{(r_0 -r_1)r_1^{2}}+\frac{1}{(r_0 -r_1)r_1}+1.
$$
Again using the properties of $\psi$ and first letting $l_1\to t_1$ in \eqref{stform4}, then taking supremum over $l_2\in[t_0-r_{1}^2,t_0)$ and noting that for such $l_2$, $\psi\equiv 1$, we also have
\begin{align}\label{stform7}
&\sup_{t_0-r_{1}^2<t<t_0}\iint_{Q^t(Z_0,r_0)} u^2(X,Y,t)\,\d X \d Y\\
&\leq\iiint_{Q_{r_0}(Z_0,r_0)} u^2\{(\partial_t+X\cdot\nabla_Y)\psi^2+4\Lambda^{3}(\psi+|\nabla_X\psi|)^2\}\,\d X \d Y \d t\notag\\
&\leq c c_{0,1}\iiint_{Q_{r_0}(Z_0,t_0)}u(X,Y,t)^2\,\d X \d Y \d t.
\end{align}
This completes the proof.
\end{proof}


\begin{lemma}\label{sub}
Let $u$ be a weak sub-solution of the equation \eqref{maineqn} in $U_X\times V_{Y,t}$ in the sense of Definition \ref{wksoldef}. Let $k\in\R$. Then $(u-k)^+$ is also a weak sub-solution of the equation \eqref{maineqn} in $U_X\times V_{Y,t}$ in the sense of Definition \ref{wksoldef}.
\end{lemma}
\begin{proof}
By Remark \ref{wksolrmk}, it is enough to prove that $u^+$ is a weak sub-solution of \eqref{maineqn}.
Let $\epsilon>0$ and $\phi\in L_{Y,t}^2(V_Y\times J,H_{X,0}^1(V_X))$ be a non-negative test function in \eqref{wksol}. Then  $\frac{u^+}{(u^+ +\epsilon)}\phi\in L_{Y,t}^2(V_Y\times J,H_{X,0}^1(V_X))$ is also a non-negative test function in \eqref{wksol}. Using $\frac{u^+}{(u^+ +\epsilon)}\phi$ as a test function in \eqref{wksol}, we obtain
\begin{align}\label{wksoltest}
&\iiint_{V_X\times V_Y\times J}A(\nabla_X u, X,Y,t)\cdot\nabla_X\phi \frac{u^+}{(u^+ +\epsilon)}\,\d X \d Y \d t\notag\\
&+\epsilon\iiint_{V_X\times V_Y\times J}A(\nabla_X u, X,Y,t)\cdot\frac {\nabla_X u^+}{(u^+ +\epsilon)^2}\phi\,\d X \d Y \d t\notag\\
&+\iint_{V_Y\times J}\ \langle (\partial_t+X\cdot\nabla_Y)u(\cdot,Y,t),\frac{u^+}{(u^+ +\epsilon)}\phi(\cdot,Y,t)\rangle\, \d Y \d t\leq 0.
\end{align}
Letting $\epsilon\to 0$, we obtain
\begin{align}\label{sub2}
&\iint_{V_Y\times J}\langle(\partial_t+X\cdot\nabla_Y)u^+(\cdot,Y,t),\phi(\cdot,Y,t)\rangle\,\d Y \d t\notag\\
&+\iiint_{V_X\times V_Y\times J}A(\nabla_X u^+,X,Y,t)\cdot\nabla_X\phi\,\d X \d Y \d t\notag\\
&+\lim\inf_{\epsilon\to 0}\epsilon\iiint_{V_X\times V_Y\times J}\frac{A(\nabla_X u^+,X,Y,t)\cdot\nabla_X u^+}{(u^+ +\epsilon)^2}\phi\,\d X \d Y \d t\leq 0.
\end{align}
However, by \eqref{assump2}-$(ii)$
\begin{align}\label{sub2gg}
\iiint_{V_X\times V_Y\times J}\frac{A(\nabla_X u^+,X,Y,t)\cdot\nabla_X u^+}{(u^+ +\epsilon)^2}\phi\,\d X \d Y \d t\geq 0.
\end{align}
Hence,
\begin{align}\label{sub2tt}
&\iint_{V_Y\times J}\langle(\partial_t+X\cdot\nabla_Y)u^+(\cdot,Y,t),\phi(\cdot,Y,t)\rangle\,\d Y \d t\notag\\
&+\iiint_{V_X\times V_Y\times J}A(\nabla_X u^+,X,Y,t)\cdot\nabla_X\phi\,\d X \d Y \d t\leq 0.
\end{align}
This proves that $u^+$ is a weak sub-solution.
\end{proof}

The following result follows from \cite[Lemma 10]{JC21}.
\begin{lemma}\label{flem}
Let $f\geq 0$ be locally integrable such that
\begin{equation}\label{flemeqn}
(\partial_t+X\cdot\nabla_Y-\Delta_X)f=\nabla_X\cdot F_1+F_2-\mu,
\end{equation}
where $F_1,F_2\in L^1\cap L^2(\R^{2m}\times \R_-)$ and $\mu\in M^1(\R^{2m}\times \R_-)$ is a non-negative measure with finite mass in $\R^{2m}\times \R_-$ such that $F_1,F_2$ and $\mu$ have compact support, in the time variable, included  in $(-\tau,0]$. Then for any $p\in[2,2+{1}/{m})$ and $\sigma\in[0,{1}/{3})$ we have
\begin{equation}\label{flem1}
\|f\|_{L^p(\R^{2m}\times \R_-)}\leq c\Big(2+\frac{1}{m}-p\Big)^{-1}(\|F_1\|_{L^2(\R^{2m}\times \R_-)}+\|F_2\|_{L^2(\R^{2m}\times \R_-)})
\end{equation}
and
\begin{align}\label{flem2}
\|f\|_{L_{t,X}^1 W_Y^{\sigma,1}(\R^{2m}\times \R_-)}\leq& c\Big(\frac{1}{3}-\sigma\Big)^{-1}(\|F_1\|_{L^1(\R^{2m}\times \R_-)}+\|F_2\|_{L^1(\R^{2m}\times \R_-)})\notag\\
&+c\Big(\frac{1}{3}-\sigma\Big)^{-1}\|\mu\|_{M^1(\R^{2m}\times \R_-)},
\end{align}
for some constant $c=c(\tau)$.
\end{lemma}

The lemmas stated so far will be sufficient for our proof of Theorem \ref{nthm1} and
Theorem \ref{locbd}.

\subsection{Additional lemmas for the proofs of Theorem \ref{thm3} and Theorem \ref{thm4}}

\begin{lemma}[Weak Poincar\'e inequality]\label{wp}
Let $\epsilon\in(0,1)$ and $\sigma\in(0,\frac{1}{3})$. Then every non-negative weak sub-solution $u$ of \eqref{maineqn} in $Q_5$ in the sense of Definition \ref{wksoldef} satisfies
\begin{equation}\label{wpine}
\Big\|(u-u_{Q_1^{-}})^+\Big\|_{L^1(Q_1^{+})}\leq c\Big(\frac{1}{\epsilon^{m+2}}\|\nabla_X u\|_{L^1(Q_5)}+\epsilon^{\sigma}\big(\frac{1}{3}-\sigma\big)^{-1}\|u\|_{L^2(Q_5)}\Big),
\end{equation}
for some constant $c=c(m,\Lambda)\geq 1$, where $Q_1^{-}:=Q_1(0,0,-1)$
and $u_{Q_1^{-}}:=\frac{1}{|Q_1^{-}|}\int_{Q_1^{-}}u$.
\end{lemma}
\begin{proof}
Using that $u$ is a non-negative weak sub-solution of \eqref{maineqn}, Theorem \ref{nthm1} and the property \ref{assump2}-$(i)$, the conclusion of the lemma follows from the lines of the proof of \cite[Proposition 13, pages 8-10]{JC21}.
\end{proof}

\begin{lemma}[Intermediate value lemma]\label{ivp}
Let $\delta_1,\delta_2\in(0,1)$ be given. Then there exists constants $\theta=c(m,\Lambda)(\delta_1\delta_2)^{10m+15}$, $r_0=\frac{1}{20}$, and $\nu\geq c(m,\Lambda)(\delta_1\delta_2)^{5m+8}$, such that  the following holds. Let $u:Q_1\to\R$ be a weak sub-solution  of \eqref{maineqn} in $Q_5$ in the sense of Definition \ref{wksoldef}, assume that $u\leq 1$ in $Q_\frac{1}{2}$, and that
\begin{equation}\label{gc1}
|\{u\leq 0\}\cap Q_{r_0}^{-}|\geq\delta_1|Q_{r_0}^{-}|\quad\text{and}\quad|\{u\geq 1-\theta\}\cap Q_{r_0}|\geq\delta_2|Q_{r_0}|,
\end{equation}
where $Q_{r_0}^{-}:=Q_{r_0}(0,0,-2r_0^{2})$. Then
\begin{equation}\label{ivpr}
\Big|\{0<u<1-\theta\}\cap Q_\frac{1}{2}\Big|\geq\nu|Q_\frac{1}{2}|.
\end{equation}
\end{lemma}
\begin{proof}
Using Lemma \ref{eng1}, Lemma \ref{sub} and Lemma \ref{wp}, the result follows from the lines of the proof of \cite[Theorem 3, pages 11-12]{JC21}.
\end{proof}

\begin{lemma}[Measure to pointwise upper bound]\label{mtp}
Given $\delta\in(0,1)$ and $r_0=\frac{1}{20}$, there exists a positive constant $\gamma:=\gamma(\delta)=c(m,\Lambda)\delta^{2(1+\delta^{-10m-16})}>0$ such that  the following holds.  Let $u$ be a weak sub-solution of \eqref{maineqn} in $Q_1$ in the sense of Definition \ref{wksoldef}, assume that $u\leq 1$ in $Q_\frac{1}{2}$ and that
\begin{equation}\label{mtpgc}
|\{u\leq 0\}\cap Q_{r_0}^{-}|\geq\delta|Q_{r_0}^{-}|,
\end{equation}
where $Q_{r_0}^{-}:=Q_{r_0}(0,0,-2r_0^{2})$. Then
$$\mbox{$u\leq 1-\gamma$ in $Q_\frac{r_0}{2}$}.$$
\end{lemma}
\begin{proof}
Using Remark \ref{wksolrmk}, Theorem \ref{locbd} and Lemma \ref{ivp}, the result follows by proceeding along the lines of the proof of \cite[Lemma 16, page 12]{JC21}.
\end{proof}

\section{Proof of Theorem \ref{nthm1}-Theorem \ref{thm4}} \label{sec4}

In this section we prove Theorem \ref{nthm1}-Theorem \ref{thm4}. We first note that since our class of operators  is closed under the  group law defined in \eqref{e70}, and by our definition of $Q_{r_0}(Z_0,t_0)$, we can throughout the second without loss of generality assume that $(Z_0,t_0)=0$. Note that
$Q_{r_0}=Q_{r_0}(0,0)=V_X\times V_Y\times J$ where $V_X=B(0,r_0)$, $V_Y=B(0,r_0^3)$, $J=(-r_0^2,0)$, and where $B(0,\rho)$ denotes the standard Euclidean ball with center at $0$ and radius $\rho$ in $\mathbb R^m$.

\subsection{Proof of Theorem \ref{nthm1}} As discussed in subsection \ref{proof}, since $u$ is a weak sub-solution of \eqref{maineqn}, there exists a non-negative measure $\bar{\mu}$ such that
$$
(\partial_t+X\cdot\nabla_Y)u=\nabla_X\cdot(A(\nabla_X u,X,Y,t))-\bar{\m}.
$$
We define $r_2:=\frac{r_0+r_1}{2}$. Let $\phi_1\in [0,1]$ be smooth such that $\phi_1\equiv 1$ in $Q_{r_1}(Z_0,t_0)$ and $\phi_1\equiv 0$ outside $Q_{r_2}(Z_0,t_0)$ satisfying
\begin{equation}\label{psi}
|\nabla_X\phi_1|\leq\frac{c}{r_0- r_2},\quad |\nabla_Y\phi_1|\leq\frac{c}{(r_0-r_{2})r_{2}^2},\quad |\partial_t\phi_1|\leq\frac{c}{(r_0-r_{2})r_{2}},
\end{equation}
for some constant $c=c(m)\geq 1$. Then we observe that $v=u\phi_1$ is a weak solution of
\begin{equation}\label{nsol}
(\partial_t+X\cdot\nabla_Y-\Delta_X)v=\nabla_X\cdot F_1+F_2-\mu\quad\text{ in }\quad\R^{N+1},
\end{equation}
where

$$
F_1=A(\nabla_X u)\phi_1-\phi_1\nabla_X u-u\nabla_X\phi_1,
$$

$$
F_2=-A(\nabla_X u)\cdot\nabla_X\phi_1+u(\partial_t+X\cdot\nabla_Y)\phi_1\quad \text{ and }\quad
\m=\bar{\m}\phi_1.
$$
By Lemma \ref{flem}, we have
\begin{equation}\label{flemap1}
\|v\|_{L^q(\R^{2m}\times \R_-)}\leq c\Big(2+\frac{1}{m}-q\Big)^{-1}(\|F_1\|_{L^2(\R^{2m}\times \R_-)}+\|F_2\|_{L^2(\R^{2m}\times \R_-)})
\end{equation}
and
\begin{equation}\label{flemap2}
\|v\|_{L_{t,X}^1 W_Y^{s,1}(\R^{2m}\times \R_-)}\leq c\Big(\frac{1}{3}-s\Big)^{-1}(\|F_1\|_{L^1(\R^{2m}\times \R_-)}+\|F_2\|_{L^1(\R^{2m}\times \R_-)}+\|\mu\|_{M^1(\R^{2m}\times \R_-)}),
\end{equation}
for some uniform constant $c$ and for every $q\in[2,2+\frac{1}{m})$ and $s\in[0,\frac{1}{3})$.
Using \eqref{psi}, \eqref{assump2}-$(i)$, Lemma \ref{eng1} and that $0<r_1<r_0\leq 1$, it follows that
\begin{equation}\label{f12}
\|F_1\|_{L^2(\R^{2m}\times \R_-)}+\|F_2\|_{L^2(\R^{2m}\times \R_-)}\leq c_1,
\end{equation}
where
\begin{equation}\label{c1}
c_1=c(m,\Lambda)\Big(1+\frac{1}{r_0-r_1}\Big)\Big(1+\frac{1}{(r_0-r_1)^2}+\frac{|X_0|+r_0}{(r_0-r_1)r_1^{2}}+\frac{1}{(r_0-r_1)r_1}\Big).
\end{equation}
Using \eqref{f12} in \eqref{flemap1}, the estimate \eqref{hint2} follows.
To obtain the estimate \eqref{ngest}, let $\phi_2\in [0,1]$ be smooth such that $\phi_2\equiv 1$ in $Q_{r_2}(Z_0,t_0)$ and $\phi_2\equiv 0$ outside $Q_{r_0}(Z_0,t_0)$ satisfying \eqref{psi}. Choosing $\phi_2$ as a test function in \eqref{nsol} and proceeding similarly as in the proof of energy estimate in Lemma \ref{eng1}, we get
$$
\|\mu\|_{M^1(Q_{r_2}(Z_0,t_0))}\leq \|\phi_2\mu\|_{M^1(\R^{2m}\times \R_-)}\leq r_0^{1+2m}c_1\|u\|_{L^2(Q_{r_0}(Z_0,t_0))},
$$
where $c_1$ is given by \eqref{c1}. The last estimate, combined with \eqref{flemap2} and \eqref{f12},
yields the estimate \eqref{ngest}. \qed
\subsection{Proof of Theorem \ref{locbd}} As mentioned, we can, without loss of generality, assume that $(Z_0,t_0)=(0,0,0)$. For $n\in\N\cup\{0\}$, we define
$$
r_n=r_\infty+(r_0-r_\infty)2^{-n},\quad T_n=-r_n^{2},\quad k_n=\frac{1}{2}(1-2^{-n}),\quad u_n=(u-k_n)^{+},
$$
and
$$
A_n:=\sup_{t\in (T_n,0)}\iint_{B(0,r_n)\times B(0,r_n^3) }u_n^{2}(\cdot,\cdot,t)\,\d X \d Y.
$$
By Lemma \ref{sub} we know that $u_n$ is a weak sub-solution of \eqref{maineqn}. Thus applying Lemma \ref{eng1} we obtain
\begin{align}\label{thm1ap}
A_n&\leq c\,c_{n-1,n}\iiint_{Q_{r_{n-1}}}u_n^{2}\,\d X \d Y \d t,\quad \forall n\geq 1,
\end{align}
where $c=c(m,\Lambda)\geq 1$ and
\begin{equation}\label{cnn}
c_{n-1,n}:=\frac{1}{(r_{n-1}-r_n)^2}+\frac{r_{n-1}}{(r_{n-1} -r_{n})r_n^{2}}+\frac{1}{(r_{n-1} -r_{n})r_n}+1\leq \frac{2^{2n}}{r_\infty^{2}(r_0-r_\infty)^2}.
\end{equation}
Now we will estimate the integral in the right hand side of \eqref{thm1ap}. Let $q=2+\frac{1}{2m}$. By H\"older's inequality we have
\begin{equation}\label{thm1ap1}
\begin{split}
\iiint_{Q_{r_{n-1}}}u_n^{2}\,\d X \d Y \d t&\leq \Big(\iiint_{Q_{r_{n-1}}}u_n^{q}\,\d X \d Y \d t\Big)^\frac{2}{q}\Big|\{u_n>0\}\cap Q_{r_{n-1}}\Big|^{1-\frac{2}{q}}.
\end{split}
\end{equation}
Since $k_n>k_{n-1}$, we get $u_n\leq u_{n-1}$. Using this fact, that $0<r_\infty<r_0\leq 1$, and Theorem \ref{nthm1}, we get
\begin{equation}\label{hghap1}
\begin{split}
\Big(\iiint_{Q_{r_{n-1}}}u_{n}^q\,\d X \d Y \d t\Big)^\frac{2}{q}&\leq \Big(\iiint_{Q_{r_{n-1}}}u_{n-1}^q\,\d X \d Y \d t\Big)^\frac{2}{q}\\
&\leq c^2\,A_{n-2}\\
&\leq\Big(\frac{c(m,\Lambda)2^{3n}}{r_\infty^{2}(r_0-r_\infty)^3}\Big)^2\,A_{n-2},
\end{split}
\end{equation}
for every $n\geq 2$, where we have used that
$$
c=c(m,\Lambda)\Big(1+\frac{1}{r_{n-2}-r_{n-1}}\Big)c_{n-2,n-1}\leq \frac{c(m,\Lambda) 2^{3n}}{r_\infty^{2}(r_0-r_\infty)^3},
$$
with $c_{n-2,n-1}$ is as defined in \eqref{cnn}.

Next, we observe that
\begin{align*}
\iiint_{Q_{r_{n-1}}}u_{n-1}^2\,\d X \d Y \d t&\geq \iiint_{\{u_{n-1}\geq 2^{-n-1}\}\cap Q_{r_{n-1}}}u_{n-1}^2\,\d X \d Y \d t\\
&\geq 2^{-2n-2}\Big|\{u_{n-1}\geq 2^{-n-1}\}\cap Q_{r_{n-1}}\Big|.
\end{align*}
Moreover,
$$
\Big|\{u_n> 0\}\cap Q_{r_{n-1}}\Big|\leq\Big|\{u_n\geq k_n-k_{n-1}\}\cap Q_{r_{n-1}}\Big|=\Big|\{u_n\geq 2^{-n-1}\}\cap Q_{r_{n-1}}\Big|.
$$
Combining the preceding two estimates and using $0<r_\infty<r_0\leq 1$, we get
\begin{equation}\label{thm1ap2}
\Big|\{u_n>0\}\cap Q_{r_{n-1}}\Big|\leq 2^{2n+2} A_{n-1}.
\end{equation}
Using the estimates \eqref{hghap1} and \eqref{thm1ap2} in \eqref{thm1ap1}, we get
\begin{equation}\label{thm1ap3}
\iiint_{Q_{r_{n-1}}}u_n^{2}\,\d X \d Y \d t\leq c(m,\Lambda)\Big(\frac{2^{4n}}{r_{\infty}^2(r_0-r_{\infty})^3}\Big)^2 A_{n-2}^{2-\frac{2}{q}},\quad\forall n\geq 2,
\end{equation}
where we have also used that $A_{n-1}\leq A_{n-2}$. Note that the latter is true since $u_{n-1}\leq u_{n-2},\,r_{n-1}< r_{n-2}$ and $T_{n-2}<T_{n-1}$ for every $n\geq 2$.
Using \eqref{thm1ap3} in \eqref{thm1ap}, we obtain
$$
A_n\leq c(m,\Lambda)\frac{2^{12n}}{r_{\infty}^6 (r_0-r_{\infty})^8}A_{n-2}^{\alpha},
$$
where $\alpha=2-\frac{2}{q}>1$, since $q>2$. Therefore, defining $S_n:=A_{2n}$, we get
$$
S_n\leq \beta^n S_{n-1}^{\alpha}\quad\forall n\geq 1,
$$
where
$$
\beta=c(m,\Lambda)\frac{2^{24}}{r_{\infty}^6 (r_0-r_{\infty})^8}.
$$
Recursively we get
\begin{equation}\label{itsc}
\begin{split}
S_n&\leq \beta^{n+(n-1)\alpha+\ldots+\alpha^{n-1}}S_1^{\alpha^{n-1}}\\
&\leq \Big(\beta^\frac{\alpha^2}{(\alpha-1)^2}S_1\Big)^{\alpha^{n-1}}\\
&\leq \Big(c(m,\Lambda)c_{0,1}\beta^\frac{\alpha^2}{(\alpha-1)^2}\|u\|^{2}_{L^2(Q_{r_0})}\Big)^{\alpha^{n-1}},
\end{split}
\end{equation}
where we have used \eqref{thm1ap} and the estimate
$$
n+\alpha(n-1)+\ldots+\alpha^{n-1}\leq\frac{\alpha^{n+1}}{(\alpha-1)^2}.
$$
Let
$$
v:=\frac{1}{\sqrt{2c(m,\Lambda)c_{0,1}\,\beta^\frac{\alpha^2}{(\alpha-1)^2}}}\frac{u}{\|u\|_{L^2(Q_{r_0})}}.
$$
We observe that
$$
\gamma:=c(m,\Lambda)c_{0,1}\,\beta^\frac{\alpha^2}{(\alpha-1)^2}\|v\|_{L^2(Q_{r_0})}^2=\frac{1}{2}<1.
$$
Note that, by the property (ii) in Remark \ref{wksolrmk}, $v$ is again a weak sub-solution of \eqref{maineqn}. Thus the estimate \eqref{itsc} holds by replacing $u$ with $v$. This fact combined with $\gamma<1$ gives $v\leq \frac{1}{2}$ a.e. in $Q_{r_\infty}$. As a consequence we get
$$
\sup_{Q_{r_\infty}}\,u\leq {\sqrt{2c(m,\Lambda)c_{0,1}\,\beta^\frac{\alpha^2}{(\alpha-1)^2}}}\, {\|u\|_{L^2(Q_{r_0})}}\leq c\Big(\frac{1}{r_{\infty}^2 (r_0-r_{\infty})^3}\Big)^\frac{\theta}{2}{\|u\|_{L^2(Q_{r_0})}},
$$
for some $c=c(m,\Lambda)\geq 1$ and $\theta=\theta(m)>1$.
Now, arguing similarly as in the proof of \cite[Proposition 12, pages 7-8]{JC21}, the result follows. \qed

\subsection{Proof of Theorem \ref{thm3}} Using Remark \ref{wksolrmk}, Theorem \ref{locbd} along with Lemma \ref{mtp}, and following the lines of the proof of \cite[Theorem 5, pages 13-14]{JC21}, the result follows. \qed

\subsection{Proof of Theorem \ref{thm4}} Using Remark \ref{wksolrmk}, Lemma \ref{mtp}, and following the lines of the proof of \cite[Theorem 7, pages 14-15]{JC21}, the result follows. \qed

\section{Proof of Theorem \ref{weakdp1}} \label{sec5}
The purpose of the section is to prove Theorem \ref{weakdp1}. As $g\in W(U_X\times V_{Y,t})$ we can in the following assume, without loss of generality, that $g\equiv 0$.

 In domains of the form $U_X\times U_Y\times I$ instead of $U_X\times V_{Y,t}$, one may attempt different approaches to prove Theorem \ref{weakdp1}, and perhaps the most natural first approach is  to add the term $\epsilon\Delta_Y$ to the operator and to instead consider the problem
\begin{equation}\label{dpweakvisbonn}
\begin{cases}
	\nabla_X\cdot(A(\nabla_Xu_\epsilon,X,Y,t))+\epsilon\Delta_Yu_\epsilon-(\partial_t+X\cdot\nabla_Y)u_\epsilon =g^*  &\text{in} \ U_X\times U_Y\times I, \\
u_\epsilon=0 & \text{on} \ \partial_p(U_X\times U_Y\times I).
\end{cases}
\end{equation}
Here $\partial_p(U_X\times U_Y\times I)$ is now the (standard) parabolic boundary of $U_X\times U_Y\times I$ , i.e.,
$$\partial_p(U_X\times U_Y\times I):=(\partial(U_X\times U_Y)\times \overline{I})\cup ((U_X\times U_Y)\times \{0\}).$$
The existence and uniqueness of weak solutions to \eqref{dpweakvisbonn} is classical and one easily deduces that
\begin{align}
&\||\nabla_Xu_\epsilon|\|^2_{L^2(U_X\times U_Y\times I)}+\epsilon \||\nabla_Yu_\epsilon|\|^2_{L^2(U_X\times U_Y\times I)}\notag\\
&\leq c
\|g^*\|_{L_{Y,t}^2(U_Y\times I,{H}_X^{-1}(U_X))}\times\||u_\epsilon|+|\nabla_Xu_\epsilon|\|_{L^2(U_X\times U_Y\times I)},
\end{align}
for some positive constant $c$, independent of $\epsilon$. By the standard Poincar{\'e} inequality, applied on $U_X$ to $u_\epsilon(\cdot,Y,t)$ with $(Y,t)$ fixed, we have
\begin{align}
\|u_\epsilon\|_{L^2(U_X\times U_Y\times I)}\leq c\||\nabla_Xu_\epsilon|\|_{L^2(U_X\times U_Y\times I)}.
\end{align}
Hence, using Cauchy-Schwarz we can conclude that
\begin{align}
&\|u_\epsilon\|^2_{L^2(U_X\times U_Y\times I)}+\||\nabla_Xu_\epsilon|\|^2_{L^2(U_X\times U_Y\times I)}+\epsilon \||\nabla_Yu_\epsilon|\|^2_{L^2(U_X\times U_Y\times I)}\notag\\
&\leq
c\|g^*\|^2_{L_{Y,t}^2(U_Y\times I,{H}_X^{-1}(U_X))},
\end{align}
for a constant $c$ which is independent of $\epsilon$. The idea is then  to let
$\epsilon\to 0$ and in this way construct a solution to the problem in \eqref{weak3}. To make this operational, already in the linear case, $A(\xi,X,Y,t)=A(X,Y,t)\xi$, one seems to need some uniform estimates up to the Kolmogorov boundary $\partial_\K(U_X\times U_Y\times I)$ to get a solution in the limit. In addition, in the nonlinear case considered in this paper we also need to ensure that $\nabla_Xu_\epsilon\to \nabla_Xu$ pointwise a.e as $\epsilon\to 0$ and how to achieve this is even less clear. One approach is to try to adapt the techniques of Boccardo and Murat \cite{BM92} but it seems unclear how to make this approach operational in our case due to the presence of the term $\epsilon\Delta_Yu_\epsilon$ in the approximating equation.

In this paper we will instead prove Theorem \ref{weakdp1} by using a variational approach recently explored in Albritton-Armstrong-Mourrat-Novack \cite{AAMN} and Litsg{\aa}rd-Nystr{\"o}m \cite{LN}. We will prove that the solution to \eqref{dpweak+} can be obtained as the minimizer of a uniformly convex functional. The fact that a parabolic equation can be cast as the first variation of a uniformly convex integral functional was first discovered by Brezis-Ekeland \cite{BE2,BE1} and for a modern treatment of this approach, covering uniformly elliptic parabolic equations of second order in the more general context of uniformly monotone operators, we refer to \cite{ABM}  which in turn is closely related to \cite{ghoussoub-tzou},  see also \cite{GBook}.

\subsection{Variational representation of the symbol}
To make the approach operational we will use a variational representation of the mapping $\xi \mapsto  A(\xi,X,Y,t)$, for each $(X,Y,t)\in \R^{N+1}$, that we learned from \cite{ABM} and \cite{AM16} and we refer to these papers for more  background. Indeed, by \cite[Theorem~2.9]{AM16}, there exists $\tilde A \in L^\infty_{\mathrm{loc}}(\mathbb R^m\times \mathbb R^m\times \R^{N+1})$ satisfying the following properties, for $\Gamma := 2\Lambda + 1$ and for each $(X,Y,t)\in \R^{N+1}$. First, the mapping
\begin{equation}
\label{e.A.unif.convex}
(\xi,\eta) \mapsto \tilde A(\xi,\eta,X,Y,t) - \frac 1 {2\Gamma}(|\xi|^2 + |\eta|^2) \quad \text{is convex}.
\end{equation}
Second, the mapping
\begin{equation}
\label{e.A.C11}
(\xi,\eta) \mapsto \tilde A(\xi,\eta,X,Y,t) -  \frac {\Gamma} 2(|\xi|^2 + |\eta|^2)  \quad \text{is concave}.
\end{equation}
Third, for every $\xi,\eta \in \mathbb R^{m}$, we have
\begin{equation}
\label{e.A.ineq}
\tilde A(\xi,\eta,X,Y,t)\ge \xi\cdot \eta,
\end{equation}
and
\begin{equation}
\label{e.A.eqiff}
\tilde A(\xi,\eta,X,Y,t) = \xi\cdot \eta
\iff
\eta= A(\xi,X,Y,t).
\end{equation}
Note that the choice of $\tilde A$ is in general not unique. Note also that \eqref{e.A.unif.convex} and \eqref{e.A.C11} imply, in particular that
\begin{align}\label{conseq}
\frac 1{2\Gamma}|\xi_1-\xi_2|^2&\leq \frac 1 2\tilde A(\xi_1,\eta,X,Y,t)+\frac 1 2\tilde A(\xi_2,\eta,X,Y,t)\notag\\
&-\tilde A(\frac 1 2\xi_1+\frac 1 2\xi_2,\eta,X,Y,t)\leq \frac {\Gamma} 2|\xi_1-\xi_2|^2.
\end{align}

\subsection{Setting up the argument} To ease the notation we will in the following at instances use the notation
$$W:=W(U_X\times V_{Y,t}),\quad W_0:=W_{0}(U_X\times V_{Y,t}),$$
and we let
$$\L u:=\nabla_X\cdot(A(\nabla_X u,X,Y,t))-(\partial_t+X\cdot\nabla_Y)u.$$
Given an arbitrary pair $(f,\j)$ such that
 \begin{eqnarray}\label{weak1-a}
\quad f\in L_{Y,t}^2(V_{Y,t},H_X^1(U_X))\quad\mbox{ and }\quad  \j\in L^2(V_{Y,t},L^2(U_X)))^m,
    \end{eqnarray}
    we introduce
\begin{equation}  \label{e.onnb}
\J[f,\j] := \iiint_{U_X\times V_{Y,t}} (\tilde A(\nabla_X f,\j,X,Y,t)-\nabla_Xf\cdot\j) \, \d X \d Y \d t.
\end{equation}
Using this notation, and given an arbitrary pair $(f,f^\ast)$ such that
 \begin{eqnarray}\label{weak1-b}
\quad\quad\quad f\in L_{Y,t}^2(V_{Y,t},H_X^1(U_X))\quad\mbox{ and }\quad f^\ast,\ f^\ast+(\partial_t+X\cdot\nabla_Y)f\in  L_{Y,t}^2(V_{Y,t},H_X^{-1}(U_X)),
    \end{eqnarray}
     we set
     \begin{equation}
\label{e.def.J}
J[f,f^*] := \inf \iiint_{U_X\times V_{Y,t}} \J[f,\g] \, \d X \d Y \d t,
\end{equation}
where the infimum is taken with respect to the set
\begin{equation}
\label{e.def.J+}
\bigl\{ \g  \in (L^2(V_{Y,t},L^2(U_X)))^m \mid {\nabla_X \cdot \g} = f^* +(\partial_t+X\cdot\nabla_Y)f \bigr\}.
\end{equation}
The condition
\begin{equation*}  
{\nabla_X \cdot \g}  = f^* +(\partial_t+X\cdot\nabla_Y)f,
\end{equation*}
appearing in \eqref{e.def.J+}, should be interpreted as stating that
\begin{equation}
\label{e.g.condition}
- \iiint_{U_X\times V_{Y,t}}\g  \cdot \nabla_X \phi \, \d X \d Y \d t = \iint_{V_{Y,t}} \langle f^*(\cdot,Y,t) +(\partial_t+X\cdot\nabla_Y)f(\cdot,Y,t), \phi\rangle \, \d Y\d t,
\end{equation}
for all $\phi \in L^2(V_{Y,t},H^1_{X,0}(U_X))$. Finally, for $g^*\in L_{Y,t}^2(V_{Y,t},{H}_X^{-1}(U_X))$ fixed we introduce
    \begin{equation}\label{setbonn}   \mathcal{A}(g^\ast):=\{ (f, \j) \in  W_{0}\times( L^2(V_{Y,t},L^2(U_X)))^m \mid \nabla_X\cdot\j = g^\ast+(\partial_t+X\cdot\nabla_Y)f \}.
\end{equation}

\subsection{$\J$ is uniformly convex on $\mathcal{A}(g^*)$}

\begin{lemma}\label{1stlemma}
    Let $g^*\in L_{Y,t}^2(V_{Y,t},{H}_X^{-1}(U_X))$ be fixed and let $\mathcal{A}(g^\ast)$ be the set introduced in \eqref{setbonn}. Then $\mathcal{A}(g^\ast)$ is non-empty.
\end{lemma}
\begin{proof} Take $f\in  W_{0}$  and consider the equation
\begin{align}\label{ex1}
&\Delta_Xv(X,Y,t) = (g^\ast(X,Y,t)+(\partial_t+X\cdot\nabla_Y)f(X,Y,t))\in H_X^{-1}(U_X),
\end{align}
for $\d Y\d t$-a.e $(Y,t)\in V_{Y,t}$. By the Lax-Milgram theorem this equation has a (unique) solution $v(\cdot)=v(\cdot,Y,t)\in H^1_{X,0}(U_X)$ and
\begin{align}  \label{ex2}
 ||\nabla_Xv||_{L_{Y,t}^2(V_{Y,t},L^2(U_X))}\leq c||g^\ast+(\partial_t+X\cdot\nabla_Y)f||_{L_{Y,t}^2(V_{Y,t},{H}_X^{-1}(U_X))}<\infty,
\end{align}
as $f\in W_0$. In particular,
\begin{equation}\label{nempty}
(f,\nabla_Xv)\in \mathcal{A}(g^\ast),
\end{equation} and hence $\mathcal{A}(g^\ast)$ is non-empty.
\end{proof}

\begin{lemma} \label{1stlemmahha} The functional $\J$ introduced in
\eqref{e.onnb} is uniformly convex on $\mathcal{A}(g^*)$.
\end{lemma}
\begin{proof} Note that if $(f,\j)\in \mathcal{A}(g^*)$ and $(\tilde f,\tilde\j)\in \mathcal{A}(0)$, then $(f+\tilde f,\j+\tilde\j)\in \mathcal{A}(g^*)$ and
$(f-\tilde f,\j-\tilde\j)\in \mathcal{A}(g^*)$. Consider $(f,\j)\in \mathcal{A}(g^*)$. We first consider the term
\begin{equation*} \label{e.uu}
-\iiint_{U_X\times V_{Y,t}} \nabla_Xf\cdot\j \, \d X \d Y \d t.
\end{equation*}
 We have
\begin{align}
-\iiint_{U_X\times V_{Y,t}} \nabla_Xf\cdot\j \, \d X \d Y \d t & =\iiint_{U_X\times V_{Y,t}} f\nabla_X\cdot\j \, \d X \d Y \d t\notag\\
&=\iint_{V_{Y,t}} \langle g^\ast(\cdot,Y,t)+(\partial_t+X\cdot\nabla_Y)f(\cdot,Y,t)),f(\cdot,Y,t)\rangle  \, \d Y \d t\notag\\
&=\iint_{V_{Y,t}} \langle g^\ast(\cdot,Y,t),f(\cdot,Y,t)\rangle  \, \d Y \d t\notag\\
&+\iint_{V_{Y,t}} \langle (\partial_t+X\cdot\nabla_Y)f(\cdot,Y,t)),f(\cdot,Y,t)\rangle  \, \d Y \d t.
\end{align}
Recall that $W_0=W_0(U_X\times V_{Y,t})$ is the closure in the norm of
$W(U_X\times V_{Y,t})$ of $C^\infty_{\K,0}(\overline{U_X\times V_{Y,t}})$. In particular, there exists $\{f_j\}$, $f_j\in C^\infty_{\K,0}(\overline{U_X\times V_{Y,t}})$ such that
$$||f-f_j||_W\to 0\mbox{ as }j\to \infty,$$
and consequently
$$||(\partial_t+X\cdot\nabla_Y)(f-f_j)||_{L_{Y,t}^2(V_{Y,t},{H}_X^{-1}(U_X))}\to 0\mbox{ as }j\to \infty.$$
Using this we see that
\begin{equation}  \label{apa}
\begin{split}
&\iint_{V_{Y,t}} \langle(\partial_t+X\cdot\nabla_Y)f(\cdot,Y,t),f(\cdot,Y,t)\rangle \, \d Y\d t\\
&\geq \liminf_{j\to\infty}
\iint_{V_{Y,t}} \langle(\partial_t+X\cdot\nabla_Y)f_j(\cdot,Y,t),f_j(\cdot,Y,t)\rangle \, \d Y\d t.
\end{split}
\end{equation}
However, using that $f_j\in C^\infty_{\K,0}(\overline{U_X\times V_{Y,t}})$  we see that
\begin{equation} \label{apa+}
\begin{split}
&\iint_{V_{Y,t}} \langle(\partial_t+X\cdot\nabla_Y)f_j(\cdot,Y,t),f_j(\cdot,Y,t)\rangle \, \d Y\d t\\
&=\iiint_{U_X\times V_{Y,t}} (\partial_t+X\cdot\nabla_Y)f_j f_j \, \d X\d Y\d t\\
&=\frac 1 2\iiint_{U_X\times V_{Y,t}} (\partial_t+X\cdot\nabla_Y)f_j^2 \, \d X\d Y\d t\\
&=\frac 12 \int_{U_X}\iint_{\partial V_{Y,t}} f_j^2(X,1)\cdot N_{Y,t} \, \d \sigma_{Y,t}\d X\geq 0,
\end{split}
\end{equation}
by the divergence theorem and the definition of the Kolmogorov boundary. Hence,
\begin{align}
-\iiint_{U_X\times V_{Y,t}} \nabla_Xf\cdot\j \, \d X \d Y \d t & =\iiint_{U_X\times V_{Y,t}} f\nabla_X\cdot\j \, \d X \d Y \d t\notag\\
&\geq \iint_{V_{Y,t}} \langle g^\ast(\cdot,Y,t),f(\cdot,Y,t)\rangle  \, \d Y \d t.
\end{align}
Using this, and observing that,
\begin{align}
&-\frac 1 2\iiint_{U_X\times V_{Y,t}} \bigl (\nabla_X(f+\tilde f)\cdot(\j+\tilde\j) +\nabla_X(f-\tilde f)\cdot(\j-\tilde\j)-2\nabla_Xf\cdot\j\bigr )\, \d X \d Y \d t\notag\\
&=-\iiint_{U_X\times V_{Y,t}} \nabla_X\tilde f\cdot\tilde\j\, \d X \d Y \d t,
\end{align}
we can conclude that
\begin{align}
&-\frac 1 2\iiint_{U_X\times V_{Y,t}} \bigl (\nabla_X(f+\tilde f)\cdot(\j+\tilde\j) +\nabla_X(f-\tilde f)\cdot(\j-\tilde\j)-2\nabla_Xf\cdot\j\bigr )\, \d X \d Y \d t\geq 0,
\end{align}
over the set $\mathcal{A}(g^\ast)$.  Hence it suffices to  prove that
\begin{equation*}  
 \iiint_{U_X\times V_{Y,t}} \tilde A(\nabla_Xf,\j,X,Y,t)\, \d X \d Y \d t
\end{equation*}
is uniformly convex over the set $\mathcal{A}(g^\ast)$. With $(f,\j)\in \mathcal{A}(g^*)$ and $(\tilde f,\tilde\j)\in \mathcal{A}(0)$ as above,   \eqref{e.A.unif.convex} implies that
\begin{equation*}  
\frac 1 2 \tilde A(\nabla_X (f+\tilde f),\j + \tilde\j,\cdot) + \frac 1 2 \tilde A(\nabla_X (f-\tilde f),\j - \tilde \j,\cdot) - \tilde A(\nabla_X f,\j,\cdot) \ge \frac 1 {2\Gamma} \Ll( |\nabla_X\tilde f|^2 + |\tilde\j|^2 \Rr) .
\end{equation*}
We also have
\begin{align*}  
\|(\partial_t+X\cdot\nabla_Y) \tilde f\|_{L_{Y,t}^2(V_{Y,t},H^{-1}(U_X))}  \le \|\tilde \j\|_{L^2(U_X\times V_{Y,t})}.
\end{align*}
Thus
\begin{align*}  
& \iiint_{U_X\times V_{Y,t}}\biggl (\frac 1 2 \tilde A(\nabla_X (f+\tilde f),\j + \tilde\j,\cdot) + \frac 1 2 \tilde A(\nabla_X (f-\tilde f),\j - \tilde \j,\cdot) - \tilde A(\nabla_X f,\j,\cdot)\biggr )\, \d X \d Y \d t\notag\\
&\geq \frac{1}{4\Gamma} \Ll( ||\nabla_X\tilde f||^2_{L^2(U_X\times V_{Y,t})}+\|(\partial_t+X\cdot\nabla_Y) \tilde f\|^2_{L_{Y,t}^2(V_{Y,t},H^{-1}(U_X))}  +
||\tilde \j||^2_{L^2(U_X\times V_{Y,t})} \Rr)\notag\\
&\geq \frac{1}{4c\Gamma} \Ll( ||\tilde f||^2_{W(U_X\times V_{Y,t})}+\||\tilde \j||^2_{L^2(U_X\times V_{Y,t})} \Rr),
\end{align*}
by using the (standard) Poincar{\'e} inequality. Hence $\J$  is uniformly convex on $\mathcal{A}(g^*)$.
\end{proof}

\subsection{Correspondence between weak solutions and minimizers} As the functional $\J$ is uniformly convex over $\mathcal{A}(g^\ast)$ there exists a unique minimizing pair $(f_1,\j_1)\in \mathcal{A}(g^\ast)$ such that
\begin{align*}  
(f_1,\j_1):=&\argmin_{(f,\j)\in \mathcal{A}(g^\ast)} \J[f,\j]\notag\\
=&\argmin_{(f,\j)\in \mathcal{A}(g^\ast)}  \iiint_{U_X\times V_{Y,t}} (\tilde A(\nabla_Xf,\j,X,Y,t)-\nabla_Xf\cdot\j) \, \d X \d Y \d t.
\end{align*}
Note that
\begin{align*}  
\min_{(f,\j)\in \mathcal{A}(g^\ast)} \J[f,\j]=\min_{f\in W_0} J[f,g^*].
\end{align*}
Moreover, by construction of $\tilde A$, see \eqref{e.A.ineq}, we have
\begin{equation} \label{e.bonna}
J[f_1,g^\ast] \ge 0.
\end{equation}

\begin{lemma}\label{corre} There is a one-to-one correspondence between weak solutions in the sense of equation \eqref{weak3} to $\L  u=g^\ast$ in $U_X\times V_{Y,t}$, such that $u\in W_0$,  and null minimizers of $J[\cdot,g^\ast]$.
\end{lemma}
\begin{proof} To prove the lemma we need to prove that for every $f \in W_{0}$, we have
\begin{align*}  
f \mbox{ solves $\L  u=g^\ast$ in the weak sense in $U_X\times V_{Y,t}$} \iff J[f,g^\ast] = 0.
\end{align*}
Indeed, the implication ''$\implies$'' is clear since if $f$ solves $\L  u=g^\ast$ in the weak sense, then
\begin{equation*}  
(f,A(\nabla_X f,X,Y,t)) \in \mathcal{A}(g^\ast) \quad \text{ and } \quad \J[f,A(\nabla_X f,X,Y,t)] = 0=J[f,g^\ast].
\end{equation*}
Conversely, if $J[f,g^\ast] = 0$, then $f = f_1$ and
\begin{equation}\label{e.null.minha}
\J[f_1,\j_1]= \iiint_{U_X\times V_{Y,t}} (\tilde A(\nabla_Xf_1,\j_1,X,Y,t)-\nabla_Xf_1\cdot\j_1) \, \d X \d Y \d t=0.
\end{equation}
Using  \eqref{e.A.eqiff}, we see that the identity \eqref{e.null.minha} implies that
\begin{equation*}  
\j_1 = A(\nabla f_1,\cdot,\cdot,\cdot) \quad \text{a.e. in } U_X\times V_{Y,t},
\end{equation*}
and by the definition of the set $\mathcal{A}(g^\ast)$,
\begin{equation*}  
\nabla_X\cdot\j_1 = g^\ast+(\partial_t+X\cdot\nabla_Y)f_1.
\end{equation*}
Hence $f_1$ indeed solves
\begin{equation*}  
\nabla_X\cdot A(\nabla f_1,\cdot,\cdot,\cdot) -(\partial_t+X\cdot\nabla_Y)f_1=g^\ast
\end{equation*}
in the weak sense. I.e., we recover that $f = f_1$ is indeed a weak solution of $\L  u=g^\ast$. In particular, the fact that there is at most one solution to $\L  u=g^*$ is clear.
\end{proof}

\subsection{An associated perturbed convex minimization problem} Using \eqref{e.bonna} and Lemma \ref{corre} we see that to complete the proof of Theorem \ref{weakdp1} it remains to prove that
\begin{equation}
\label{e.null.min}
J[f_1,g^*] \le 0.
\end{equation}
In order to do so, we introduce the perturbed convex minimization problem defined, for every $f^* \in L_{Y,t}^2(V_{Y,t},H_X^{-1}(U_X))$, by
\begin{equation*}  
G(f^*) := \inf_{f \in W_{0}}\bigl ( J[f, f^*+g^* ]  {- \iint_{ V_{Y,t}} \langle f^*(\cdot,Y,t),f(\cdot,Y,t)\rangle\, \d Y\d t\bigr ).}
\end{equation*}
As \begin{equation*}  
G(0) = \inf_{f \in W_{0}} J[f, g^* ],
\end{equation*}
we see that to prove \eqref{e.null.min} is suffices to prove that $G(0) \le 0$.

\begin{lemma} $G$ is a convex, locally bounded from above and lower semi-continuous functional on $L_{Y,t}^2(V_{Y,t},H_X^{-1}(U_X))$.
\end{lemma}
\begin{proof} For every pair  $(f,\j) \in \mcl A(f^*+g^*)$, we have
\begin{equation*}
\nabla_X \cdot \j = f^*+g^* +(\partial_t+X\cdot\nabla_Y)f,
\end{equation*}
and thus
\begin{align*}  
\J[f,\j] & = \iiint_{U_X\times V_{Y,t}} (\tilde A(\nabla_Xf,\j,X,Y,t)-\nabla_Xf\cdot\j) \, \d X \d Y \d t\notag\\
 & = \iiint_{U_X\times V_{Y,t}} \tilde A(\nabla_Xf,\j,X,Y,t) \, \d X \d Y \d t{+\iint_{V_{Y,t}}   \langle (f^*+g^*)(\cdot,Y,t),f(\cdot,Y,t)\rangle\, \d Y\d t}\notag\\
 &+\iint_{V_{Y,t}}   \langle (\partial_t+X\cdot\nabla_Y)f(\cdot,Y,t),f(\cdot,Y,t)\rangle\, \d Y\d t.
\end{align*}
Hence
\begin{align*}  
&\J[f,\j] {-\iint_{V_{Y,t}} \langle f^*(\cdot,Y,t),f(\cdot,Y,t)\rangle\, \d Y\d t} \notag\\
& = \iiint_{U_X\times V_{Y,t}} \tilde A(\nabla_Xf,\j,X,Y,t) \, \d X \d Y \d t+\iint_{V_{Y,t}}   \langle  g^*(\cdot,Y,t),f(\cdot,Y,t)\rangle\, \d Y\d t\\
&+\iint_{V_{Y,t}}   \langle (\partial_t+X\cdot\nabla_Y)f(\cdot,Y,t),f(\cdot,Y,t)\rangle\, \d Y\d t.
\end{align*}
Taking the infimum over all $(f,\j)$ satisfying the affine constraint $(f,\j) \in \mcl A(f^*+g^*)$ we obtain the quantity $G(f^*)$, i.e., $G(f^*)$ can be expressed as
\begin{equation*}  
G(f^*) = \inf_{(f,\j):\  (f,\j) \in \mcl A(f^*+g^*)}\bigl ( \J[f,\j] { - \iint_{V_{Y,t}} \langle f^*(\cdot,Y,t),f(\cdot,Y,t)\rangle\, \d Y\d t} \bigr ).
\end{equation*}
In particular, $G(f^*)$ can be expressed as the infimum of
\begin{align}\label{lulu}
 &\iiint_{U_X\times V_{Y,t}} \tilde A(\nabla_Xf,\j,X,Y,t) \, \d X \d Y \d t+\iint_{V_{Y,t}}   \langle  g^*(\cdot,Y,t),f(\cdot,Y,t)\rangle\, \d Y \d t\notag\\
 &+\iint_{V_{Y,t}}   \langle (\partial_t+X\cdot\nabla_Y)f(\cdot,Y,t),f(\cdot,Y,t)\rangle\, \d Y\d t
\end{align}
with respect to $(f,\j)$ such that $(f,\j) \in \mcl A(f^*+g^*)$. We now recall the argument in \eqref{apa} and \eqref{apa+}. In particular, given $f\in W_0$ there exists $\{f_j\}$, $f_j\in C^\infty_{\K,0}(\overline{U_X\times V_{Y,t}})$ such that
\begin{equation}  \label{apauu-}||f-f_j||_W\to 0\mbox{ as }j\to \infty,
\end{equation}
and consequently
$$||(\partial_t+X\cdot\nabla_Y)(f-f_j)||_{L_{Y,t}^2(V_{Y,t},{H}_X^{-1}(U_X))}\to 0\mbox{ as }j\to \infty.$$
Using \eqref{apa} and \eqref{apa+} we have
\begin{equation}  \label{apauu}
\begin{split}
&\iint_{V_{Y,t}} \langle(\partial_t+X\cdot\nabla_Y)f(\cdot,Y,t),f(\cdot,Y,t)\rangle \, \d Y\d t\notag\\
&=\lim_{j\to\infty}
\frac 12 \int_{U_X}\iint_{\partial V_{Y,t}} f_j^2|(X,1)\cdot N_{Y,t}| \, \d \sigma_{Y,t}\d X.
\end{split}
\end{equation}
Obviously we get the same limit in \eqref{apauu} independent of what sequence $\{f_j\}$ chosen as long as \eqref{apauu-} holds. Now consider
$f,g\in W_0$ and let $\{f_j\}$, $f_j\in C^\infty_{\K,0}(\overline{U_X\times V_{Y,t}})$, $\{g_j\}$, $g_j\in C^\infty_{\K,0}(\overline{U_X\times V_{Y,t}})$, be such that
\begin{equation}  \label{apauu-+}||f-f_j||_W+||g-g_j||_W\to 0\mbox{ as }j\to \infty,
\end{equation}
Then
\begin{equation}  \label{apauu-++}||(\tau f+(1-\tau)g)-(\tau f_j+(1-\tau)g_j)||_W\to 0\mbox{ as }j\to \infty,
\end{equation}
for all $\tau\in [0,1]$. Hence
\begin{align}  \label{apauu+++++}
&\iint_{V_{Y,t}} \langle(\partial_t+X\cdot\nabla_Y)(\tau f+(1-\tau)g)(\cdot,Y,t),(\tau f+(1-\tau)g)(\cdot,Y,t)\rangle \, \d Y\d t\notag\\
&=\lim_{j\to\infty}
\frac 12 \int_{U_X}\iint_{\partial V_{Y,t}} (\tau f_j+(1-\tau)g_j)^2|(X,1)\cdot N_{Y,t}| \, \d \sigma_{Y,t}\d X\notag\\
&\leq\lim_{j\to\infty}
\frac 12 \int_{U_X}\iint_{\partial V_{Y,t}} \tau f_j^2|(X,1)\cdot N_{Y,t}| \, \d \sigma_{Y,t}\d X\notag\\
&+\lim_{j\to\infty}
\frac 12 \int_{U_X}\iint_{\partial V_{Y,t}} (1-\tau)g_j^2|(X,1)\cdot N_{Y,t}| \, \d \sigma_{Y,t}\d X,
\end{align}
and we deduce that
\begin{align}  \label{apauu+++++1}
&\iint_{V_{Y,t}} \langle(\partial_t+X\cdot\nabla_Y)(\tau f+(1-\tau)g)(\cdot,Y,t),(\tau f+(1-\tau)g)(\cdot,Y,t)\rangle \, \d Y\d t\notag\\
&\leq\tau\iint_{V_{Y,t}} \langle(\partial_t+X\cdot\nabla_Y)f(\cdot,Y,t),f(\cdot,Y,t)\rangle \, \d Y\d t\notag\\
&+(1-\tau)\iint_{V_{Y,t}} \langle(\partial_t+X\cdot\nabla_Y)g(\cdot,Y,t),g(\cdot,Y,t)\rangle \, \d Y\d t.
\end{align}
In particular, we can conclude that the mapping
$$f\to \iint_{V_{Y,t}} \langle(\partial_t+X\cdot\nabla_Y)(\tau f+(1-\tau)g)(\cdot,Y,t),(\tau f+(1-\tau)g)(\cdot,Y,t)\rangle \, \d Y\d t$$
is convex on $W_0$.  Using this, and \eqref{e.A.unif.convex}, we see that the expression in \eqref{lulu} is convex as a function of $(f,f^*,\j)$ and this proves that $G$ is convex. Furthermore, using \eqref{ex1}, \eqref{ex2}, and  \eqref{nempty} we can conclude that the {infimum of the expression in} \eqref{lulu} is finite, hence $G(f^*)<\infty$. In particular, the function $G$ is locally bounded from above. These two properties imply that $G$ is lower semi-continuous, see \cite[Lemma~2.1 and Corollary~2.2]{ET}.
\end{proof}

\subsection{The convex dual of $G$} We denote by $G^*$ the convex dual of $G$, defined for every
$$h \in (L_{Y,t}^2(V_{Y,t},H_X^{-1}(U_X)))^\ast=L_{Y,t}^2(V_{Y,t},H_{X,0}^1(U_X)),$$ as
\begin{equation*}  
G^*(h) := \sup_{f^* \in L_{Y,t}^2(V_{Y,t},H_X^{-1}(U_X))} \bigl( -G(f^*) + \iint_{V_{Y,t}} \langle f^*(\cdot,Y,t),h(\cdot,Y,t)\rangle\, \d Y\d t \bigr).
\end{equation*}
Let $G^{**}$ be the bidual of $G$. Since $G$ is lower semi-continuous, we have that $G^{**} = G$ (see \cite[Proposition~4.1]{ET}), and in particular,
\begin{equation*}  
G(0) = G^{**}(0) = \sup_{h \in L_{Y,t}^2(V_{Y,t},H_{X,0}^1(U_X))} \bigl( -G^*(h) \bigr) .
\end{equation*}
In order to prove that $G(0) \le 0$, it therefore suffices to show that
\begin{equation}
\label{e.get.to.nullmin}
G^*(h) \ge 0\mbox{ for all } h \in L_{Y,t}^2(V_{Y,t},H_{X,0}^1(U_X)).
\end{equation}

To continue we note that we can rewrite $G^*(h)$ as
\begin{equation}
\label{e.second.G*}
\begin{split}
G^*(h) = \sup_{(f,\j,f^*)} &\bigg\{ \iiint_{U_X\times V_{Y,t}} -(\tilde A(\nabla_X f,\j,\cdot,\cdot,\cdot)-(\nabla_X f\cdot \j)) \,  \d X \d Y \d t\\
&+\iint_{V_{Y,t}} {\langle f^*(\cdot,Y,t),(h(\cdot,Y,t)+f(\cdot,Y,t))\rangle}\, \d Y \d t\bigg \},
\end{split}
\end{equation}
where the supremum is taken with respect to $$(f,\j,f^*)\in W_{0}\times (L_{Y,t}^2(V_{Y,t},L^2_X(U_X)))^m\times L_{Y,t}^2(V_{Y,t},H_X^{-1}(U_X)),$$ subject to the constraint
\begin{equation}
\label{e.constraint.fj}
\nabla_X\cdot \j = f^* + g^* +(\partial_t+X\cdot\nabla_Y)f.
\end{equation}
Furthermore, note that  for every $h \in  L_{Y,t}^2(V_{Y,t},H_{X,0}^1(U_X))$, we have $G^*(h) \in \R \cup \{+\infty\}$.

\begin{lemma}\label{red} Consider $h \in  L_{Y,t}^2(V_{Y,t},H_{X,0}^1(U_X))$. Then
\begin{equation}
\label{e.finite.G*}
G^*(h) < +\infty \quad \implies \quad h \in W\cap L_{Y,t}^2(V_{Y,t},H_{X,0}^1(U_X)).
\end{equation}
\end{lemma}
\begin{proof}  To prove the lemma we need to prove that $(\partial_t+X\cdot\nabla_Y)h \in L_{Y,t}^2(V_{Y,t},H_{X}^{-1}(U_X))$.  Using that we take a supremum in the definition of $G^\ast$ we can develop lower bounds on $G^\ast$ by restricting the set with respect to which we take the supremum. Here, for $f\in W_0$, we choose to restrict the supremum  to $(f,\j,f^*)$ where  $\j=\j_0$ is a solution of $\nabla_X\cdot \j_0= g^* $ and $f^* := -(\partial_t+X\cdot\nabla_Y)f$. Recall from \eqref{nempty} that such a $\j_0 \in (L_{Y,t}^2(V_{Y,t},L^2_X(U_X)))^m$ exists.  With these choices for $\j$ and  $f^*$, the constraint \eqref{e.constraint.fj} is satisfied, and we obtain that
\begin{align*}  
G^*(h)\geq \sup_{f\in W_0} &\biggl\{ \iiint_{U_X\times V_{Y,t}} -(\tilde A(\nabla_X f,\j_0,\cdot,\cdot,\cdot)-(\nabla_X f\cdot \j_0)) \, \d X \d Y \d t\notag\\
&-\iint_{V_{Y,t}} {\langle (\partial_t+X\cdot\nabla_Y)f(\cdot,Y,t),(h(\cdot,Y,t) + f(\cdot,Y,t))\rangle}\, \d Y\d t\biggr \}.
\end{align*}
Consider  $f\in  C^\infty_{\K,0}(\overline{U_X\times V_{Y,t}})\subset W_0$. Then, again arguing as in \eqref{apa}, \eqref{apa+},
\begin{align*}  
-\iint_{V_{Y,t}} \langle (\partial_t+X\cdot\nabla_Y)f(\cdot,Y,t),f(\cdot,Y,t)\rangle\, \d Y\d t\leq 0.
\end{align*}
{Furthermore, restricting to $f\in C^\infty_{0}({U_X\times V_{Y,t}})\subset C^\infty_{\K,0}(\overline{U_X\times V_{Y,t}})$ yields by the same argument that
\begin{align*}  
-\iint_{V_{Y,t}} \langle (\partial_t+X\cdot\nabla_Y)f(\cdot,Y,t),f(\cdot,Y,t)\rangle\, \d Y\d t = 0.
\end{align*}
}Hence we have the lower bound
\begin{align*}  
G^*(h)\geq \sup &\biggl\{ \iiint_{U_X\times V_{Y,t}} -(\tilde A(\nabla_X f,\j_0,\cdot,\cdot,\cdot)-(\nabla_X f\cdot \j_0)) \, \d X \d Y \d t\notag\\
&-\iint_{V_{Y,t}} {\langle (\partial_t+X\cdot\nabla_Y)f(\cdot,Y,t),h(\cdot,Y,t)\rangle}\, \d Y\d t\biggr \},
\end{align*}
where the supremum now is taken with respect to $f\in C_0^\infty(U_X\times V_{Y,t})\subset C^\infty_{\K,0}(\overline{U_X\times V_{Y,t}})$.
 Moreover, as $G^*(h) < +\infty$, we have that
\begin{align*}  
&-\iint_{V_{Y,t}} \langle (\partial_t+X\cdot\nabla_Y)f(\cdot,Y,t),h(\cdot,Y,t)\rangle\, \d Y\d t\notag\\
&\leq  \iiint_{U_X\times V_{Y,t}} (\tilde A(\nabla_X f,\j_0,\cdot,\cdot,\cdot)-(\nabla_X f\cdot \j_0)) \, \d X \d Y \d t +G^*(h)<\infty,
\end{align*}
for every $f\in C_0^\infty(U_X\times V_{Y,t})$ fixed. {Note that by replacing $f$ with $-f$ in the above argument we also obtain a lower bound}. In particular,
%
{$$\sup \ \biggl |\iint_{V_{Y,t}} \langle (\partial_t+X\cdot\nabla_Y)h(\cdot,Y,t),f(\cdot,Y,t)\rangle\, \d Y\d t\biggr | <\infty,$$
where the supremum is taken over $f \in C_0^\infty(U_X\times V_{Y,t})$ such that $||f||_{L_{Y,t}^2(V_{Y,t},H_{X,0}^1(U_X))}\leq 1$.} Using that $C_0^\infty(U_X\times V_{Y,t})$ is dense in $L_{Y,t}^2(V_{Y,t},H_{X,0}^1(U_X))$ we can conclude that
$$(\partial_t+X\cdot\nabla_Y)h\in L_{Y,t}^2(V_{Y,t},H_X^{-1}(U_X))$$ and this observation proves \eqref{e.finite.G*}.
\end{proof}

\subsection{Bounding $G^*$ from below} Lemma \ref{red} gives at hand that  in place of \eqref{e.get.to.nullmin}, we have reduced the matter to proving that
\begin{equation}
\qquad G^*(h) \ge 0\mbox{ for all }h\in W\cap L_{Y,t}^2(V_{Y,t},H_{X,0}^1(U_X)).
\end{equation}
Furthermore, note that for $\tilde h\in W\cap C_{X,0}^\infty(\overline{U_X\times V_{Y,t}})$ we have
\begin{equation}\label{smoothbound}
    G^*(h) \geq G^*(\tilde h) - \|f^*\|_{L^2_{Y,t}(V_{Y,t},H_X^{-1}(U_X))} \| h-\tilde h \|_{L^2_{Y,t}(V_{Y,t},H^1_{X}(U_X))}.
\end{equation}
As we are to establish a lower bound on $G^*$,  we may restrict to taking the supremum over $f^*$ such that
\begin{equation}\label{f*norm}
\|f^*\|_{L_{Y,t}^2(V_{Y,t},H_X^{-1}(U_X))}\leq 1.
\end{equation}
In Lemma \ref{red+} below we prove that
$$G^*(h) \ge 0\mbox{ for all }h\in W\cap C_{X,0}^\infty(\overline{U_X\times V_{Y,t}}).$$
By combining this with \eqref{smoothbound} and \eqref{f*norm} we see that
\[
    G^*(h) \geq G^*(\tilde h) - \| h-\tilde h \|_{L^2_{Y,t}(V_{Y,t},H^1_{X,0}(U_X))}\geq  - \| h-\tilde h \|_{L^2_{Y,t}(V_{Y,t},H^1_{X,0}(U_X))},
\]
for all $\tilde h \in W\cap C^\infty_{X,0}(\overline{U_X\times V_{Y,t}})$. Furthermore, by the definitions of $W$, and $L_{Y,t}^2(V_{Y,t},H_{X,0}^1(U_X))$, we can choose a sequence $h_j\in W\cap C^\infty_{X,0}(\overline{U_X\times V_{Y,t}})$ such that
\[
\lim_{j\rightarrow\infty} \| h-h_j \|_{L^2_{Y,t}(V_{Y,t},H^1_{X,0}(U_X))} = 0.
\]
Hence the proof that $G^*(h)\geq 0$, and hence the final piece in the proof of existence in Theorem \ref{weakdp1}, is to prove the following lemma.

\begin{lemma}\label{red+}
\begin{equation}
\label{e.get.to.nullmin2}
\qquad G^*(h) \ge 0\mbox{ for all }h\in W\cap C^\infty_{X,0}(\overline{U_X\times V_{Y,t}}).
\end{equation}
\end{lemma}
\begin{proof}
 To start the proof of the lemma we first note that we have, as $f\in W_0$,  that $$(\partial_t+X\cdot\nabla_Y)f \in L_{Y,t}^2(V_{Y,t},H_X^{-1}(U_X)),$$ and hence we can replace $f^*$ by {$f^* - (\partial_t+X\cdot\nabla_Y)f$} in the variational formula \eqref{e.second.G*} for $G^*$ to get
\begin{align*}
G^*(h) \geq \sup_{(f,\j,f^*)} &\biggl\{ \iiint_{U_X\times V_{Y,t}}  -(\tilde A(\nabla_X f,\j,\cdot,\cdot,\cdot)-(\nabla_X f\cdot \j)) \, \d X \d Y \d t\notag\\
&+\iint_{V_{Y,t}} {\langle (f^* - (\partial_t+X\cdot\nabla_Y)f)(\cdot,Y,t),(h(\cdot,Y,t) + f(\cdot,Y,t))\rangle}\, \d Y\d t\biggr \},
\end{align*}
where the supremum now is taken with respect to
\begin{align}
\label{e.rest}(f,\j,f^*)\in (W\cap C^\infty_{\K,0}(\overline{U_X\times V_{Y,t}}))\times (L_{Y,t}^2(V_{Y,t},L^2_X(U_X)))^m\times L_{Y,t}^2(V_{Y,t},H_X^{-1}(U_X)),
\end{align}
subject to the constraint
\begin{equation}
\label{e.constraint.fj2}
\nabla_X\cdot \j = f^* + g^*.
\end{equation}
Next  using that $f\in C^\infty_{\K,0}(\overline{U_X\times V_{Y,t}})$, $h\in C^\infty_{X,0}(\overline{U_X\times V_{Y,t}})$, we have
\begin{align*}  
&\iint_{V_{Y,t}} -\langle (\partial_t+X\cdot\nabla_Y)f(\cdot,Y,t),(h(\cdot,Y,t)+f(\cdot,Y,t))\rangle\notag\\
&=\iint_{V_{Y,t}} \langle (\partial_t+X\cdot\nabla_Y)h(\cdot,Y,t),f(\cdot,Y,t)\rangle\, \d Y\d t\notag\\
&\quad {-\int_{U_X}\iint_{\partial V_{Y,t}} \bigl (\frac 1 2 f^2+fh)(X,1)\cdot N_{Y,t}\, \d\sigma_{Y,t}\d X.}
\end{align*}
Using the identity in the last display we see that
\begin{align}
\label{e.G*2a}
G^*(h)\geq \sup_{(f,\j,f^*)} &\biggl\{ \iiint_{U_X\times V_{Y,t}}   -(\tilde A(\nabla_X f,\j,\cdot,\cdot,\cdot)-(\nabla_X f\cdot \j)) \, \d X \d Y \d t\notag\\
&+\iint_{V_{Y,t}} \langle f^*,(h(\cdot,Y,t)+f(\cdot,Y,t))\rangle +\langle (\partial_t+X\cdot\nabla_Y)h(\cdot,Y,t),f(\cdot,Y,t)\rangle\, \d Y\d t\notag\\
&-\int_{U_X}\iint_{\partial V_{Y,t}} \bigl (\frac 1 2 f^2+fh)(X,1)\cdot N_{Y,t}\, \d\sigma_{Y,t}\d X\biggr \},
\end{align}
where the supremum still is with respect to $(f,\j,f^*)$ as in \eqref{e.rest}  subject to \eqref{e.constraint.fj2}. Now, by arguing exactly as in the passage between displays (3.23) and (3.26) in \cite{LN}, using the properties of $\tilde A$, we can conclude that it suffices to prove that
$\tilde G^*(h)\geq 0$ where
\begin{align*}
\tilde G^*(h) := \sup_{(\tilde f,\j,f^*,b)} &\biggl\{ \iiint_{U_X\times V_{Y,t}}  -(\tilde A(\nabla_X \tilde f,\j,\cdot,\cdot,\cdot)-(\nabla_X \tilde f\cdot \j)) \, \d X \d Y \d t\notag\\
&+\iint_{V_{Y,t}} \langle f^*,(h(\cdot,Y,t)+\tilde f(\cdot,Y,t))\rangle +\langle (\partial_t+X\cdot\nabla_Y)h(\cdot,Y,t),\tilde f(\cdot,Y,t)\rangle\, \d Y\d t\notag\\
&-\int_{U_X}\iint_{\partial V_{Y,t}} \bigl (\frac 1 2 b^2+bh)(X,1)\cdot N_{Y,t}\, \d\sigma_{Y,t}\d X\biggr \},
\end{align*}
and where the supremum is taken with respect to all $(\tilde f,\j,f^*,b)$ in the set
\begin{align*}
(W\cap C^\infty_{X,0}(\overline{U_X\times V_{Y,t}}))\times (L_{Y,t}^2(V_{Y,t},L^2_X(U_X)))^m\times L_{Y,t}^2(V_{Y,t},H_X^{-1}(U_X))\times C^\infty_{\K,0}(\overline{U_X\times V_{Y,t}}),
\end{align*}
subject to the condition stated in \cite{LN}, i.e., that
$$\Gamma(\tilde f,b):=||\tilde f||_{L_{Y,t}^2(V_{Y,t},H_X^1(U_X))}+||b||_{L_{Y,t}^2(V_{Y,t},H_X^1(U_X))}\leq\Gamma$$
for some large but fixed $\Gamma\geq 1$. However, this implies that $\tilde f:=-h$ is an admissible function. With this choice of $\tilde f$, we then let $\j:= A(-\nabla_X h,X,Y,t) \in (L_{Y,t}^2(V_{Y,t},L^2_X(U_X)))^m$ and then $$f^* = \nabla_X\cdot \j -g^*\in  L_{Y,t}^2(V_{Y,t},H_X^{-1}(U_X)).$$
Using this we deduce that
\begin{align*}
\tilde G^*(h) \geq \sup_{b} &\biggl\{-\int_{U_X}\iint_{\partial V_{Y,t}} \frac 1 2 (b+h)^2(X,1)\cdot N_{Y,t}\, \d\sigma_{Y,t}\d X\biggl \},
\end{align*}
where supremum now is taken with respect to $b\in C^\infty_{\K,0}(\overline{U_X\times V_{Y,t}})$. Using Lemma \ref{lemma1c}  below  it follows that
\begin{align*}
\sup_{b} &\biggl\{-\int_{U_X}\iint_{\partial V_{Y,t}} \frac 1 2 (b+h)^2(X,1)\cdot N_{Y,t}\, \d\sigma_{Y,t}\d X\biggl \}\geq 0.
\end{align*}
The proof of the lemma is therefore complete.
\end{proof}

\begin{lemma}\label{lemma1c} Assume that $h\in W(U_X\times V_{Y,t})\cap C^\infty_{X,0}(\overline{U_X\times V_{Y,t}})$. Then
\begin{align}
\sup_{b\in W\cap C^\infty_{\K,0} (U_X\times V_{Y,t})} -\iiint_{U_X\times\partial V_{Y,t}} {(b+h)^2}(X,1)\cdot N_{Y,t}\, \d\sigma_{Y,t}\d X\geq 0.
\end{align}
\end{lemma}

Lemma \ref{lemma1c} is Lemma 3.7 in \cite{LN} and in the next subsection we supply parts of the proof for completion.

\subsection{Proof of Lemma \ref{lemma1c}} Let $\psi(s)\in C^\infty(\R)$ be such that $0\leq \psi \leq 1$,
\begin{equation*}
        \psi \equiv 1\ \mbox{on }[ 0,1],\ \psi \equiv 0\ \mbox{on }[ 2,\infty),
\end{equation*}
$|\psi'|\leq 2$ and such that $\sqrt{1-\psi^2}\in C^\infty(\R)$. Based on $\psi$ we introduce for $r$, $0\leq r <\infty$
\begin{equation}\label{bruu}
\psi_r(X,Y,t) := \psi\biggl( r\, \frac{\big((X,1)\cdot N_{Y,t}\big)^+}{1+|X|^2} \biggr),
\end{equation}
where we use the notation $s^+:=\max\lbrace s,0 \rbrace$ for $s\in\R$. As $h$ is smooth, and $U_X$ and $V_{Y,t}$ are bounded domains, we have
\begin{equation}\label{bdryfinite}
    \iiint_{U_X\times \partial V_{Y,t}} h^2|(X,1)\cdot N_{Y,t}|\d X\d \sigma_{Y,t} < \infty.
\end{equation}
Let, for any $r\geq 0$,
\begin{equation}\label{br}
    {b_r := (\psi_r-1)h.}
\end{equation}
As in the proof of Lemma 3.7 in \cite{LN} it follows that
\begin{equation}\label{bdryfinitea}
b_r\in W(U_X\times V_{Y,t}).
\end{equation}
By construction, $b_r$ vanishes on $\partial_\K(U_X\times V_{Y,t})$. Together with \eqref{bdryfinitea}, this yields that $b_r\in W\cap C^\infty_{\K,0}(\overline{U_X\times V_{Y,t}})$. Furthermore,
\begin{equation*}
    \begin{split}
        -\iiint_{U_X\times\partial V_{Y,t}} {(b_r+h)^2}(X,1)\cdot N_{Y,t}\, \d\sigma_{Y,t}\d X &= -\iiint_{U_X\times\partial V_{Y,t}} \psi_r^2h^2(X,1)\cdot N_{Y,t}\, \d\sigma_{Y,t}\d X.
    \end{split}
\end{equation*}
Letting $r\to \infty$ we see that
\begin{align*}
&\lim_{r\rightarrow\infty} -\iiint_{U_X\times\partial V_{Y,t}} \psi_r^2h^2(X,1)\cdot N_{Y,t}\, \d\sigma_{Y,t}\d X\\
&=\iiint_{U_X\times\partial V_{Y,t}} h^2\big((X,1)\cdot N_{Y,t}\big)^+\, \d\sigma_{Y,t}\d X\geq 0.\end{align*}

\subsection{The proof of Theorem \ref{weakdp1}} Retracing the argument we see by \eqref{e.bonna} and \eqref{e.null.min} that
\begin{equation}
\label{e.null.minbonnG}
J[f_1,g^*]=0\mbox{ for some }f_1\in W_0.
\end{equation}
Using Lemma \ref{corre}  we can conclude that $f_1$ is the unique weak solution $f_1\in W_0$  to $\L  u=g^\ast$ in $U_X\times V_{Y,t}$ in the sense of equation \eqref{weak3}. This completes the proof of existence and uniqueness part of Theorem \ref{weakdp1}. The quantitative estimate follows in the standard way.

\subsection{A comparison principle} Assume that $u\in W(U_X\times V_{Y,t})$ is a weak sub-solution to the equation
 \begin{equation}\label{dpweak+llhh}
	\nabla_X\cdot(A(\nabla_X u,X,Y,t))-(\partial_t+X\cdot\nabla_Y)u = g^* \text{ in } \ U_X\times V_{Y,t}.
\end{equation}
By definition this means in particular that $u\in W(U_X\times V_{Y,t})$. Given $u$ we now let
$v\in W(U_X\times V_{Y,t})$ be the unique weak solution to the problem
 \begin{equation}\label{dpweak+ll}
\begin{cases}
	\nabla_X\cdot(A(\nabla_X v,X,Y,t))-(\partial_t+X\cdot\nabla_Y)v = g^*  &\text{in} \ U_X\times V_{Y,t}, \\
      v = u & \text{on} \ \partial_\K(U_X\times V_{Y,t}),
\end{cases}
\end{equation}
in the sense that
\begin{eqnarray}\label{weak1ha}
v\in W(U_X\times V_{Y,t}),\ (v-u)\in  W_0(U_X\times V_{Y,t}),
    \end{eqnarray}
    and in the sense that \eqref{weak3} holds for all $ \phi\in L_{Y,t}^2(V_{Y,t},H_{X,0}^1(U_X))$. By Theorem \ref{weakdp1} $v$ exists and is unique. We want to prove that
    $u\leq v$ a.e in $U_X\times V_{Y,t}$. To achieve this we let $\epsilon>0$ be arbitrary and we use the test function $\phi=(u-v-\epsilon)^+$.
    Then $\phi$ is a non-negative admissible test function and $\phi=0$ on $\partial_\K(U_X\times V_{Y,t})$.  Hence,
    \begin{align}\label{wksolcc}
&\iiint_{U_X\times V_{Y,t}}A(\nabla_X u, X,Y,t)\cdot\nabla_X\phi\,\d X \d Y \d t\notag\\
&+\iint_{V_{Y,t}}\ \langle g^\ast(\cdot,Y,t)+ (\partial_t+X\cdot\nabla_Y)u(\cdot,Y,t),\phi(\cdot,Y,t)\rangle\, \d Y \d t\leq 0,
\end{align}
and
\begin{align}\label{wksolcc+}
&\iiint_{U_X\times V_{Y,t}}A(\nabla_X v, X,Y,t)\cdot\nabla_X\phi\,\d X \d Y \d t\notag\\
&+\iint_{V_{Y,t}}\ \langle g^\ast(\cdot,Y,t)+ (\partial_t+X\cdot\nabla_Y)v(\cdot,Y,t),\phi(\cdot,Y,t)\rangle\, \d Y \d t=0.
\end{align}
Subtracting these relations, we get
\begin{align}\label{wksolcc++}
&\iiint_{U_X\times V_{Y,t}}(A(\nabla_X v, X,Y,t)-A(\nabla_X u, X,Y,t))\cdot\nabla_X\phi\,\d X \d Y \d t\notag\\
&+\iint_{V_{Y,t}}\langle\partial_t+X\cdot\nabla_Y)(v-u)(\cdot,Y,t),\phi(\cdot,Y,t)\rangle\, \d Y \d t\geq 0.
\end{align}
Using the property \eqref{assump1}-$(ii)$, we now first note that
\begin{align}\label{wksolcc++a}
&\iiint_{U_X\times V_{Y,t}}(A(\nabla_X v, X,Y,t)-A(\nabla_X u, X,Y,t))\cdot\nabla_X\phi\,\d X \d Y \d t\notag\\
&\leq -\Lambda^{-1}\iiint_{U_X\times V_{Y,t}}|\nabla_X(u-v-\epsilon)^+|^2\, \d X \d Y \d t.
\end{align}
Second, again using the definition of $W(U_X\times V_{Y,t})$ and that $(v-u)\in  W_0(U_X\times V_{Y,t})$, we see that  we see that there exists a sequence $\{f_j\}$, $f_j\in C_{\K,0}^\infty(\overline{U_X\times V_{Y,t}})$ such
\begin{align}\label{wksolcc++b}
&\iint_{V_{Y,t}}(\partial_t+X\cdot\nabla_Y)(v-u)(\cdot,Y,t),\phi(\cdot,Y,t)\rangle\, \d Y \d t\notag\\
&=-\lim_{j\to\infty}
\iiint_{U_X\times V_{Y,t}}(\partial_t+X\cdot\nabla_Y)(f_j-\epsilon)^+(f_j-\epsilon)^+\, \d X \d Y \d t\notag\\
&=-\frac 12\lim_{j\to\infty}
\iint_{U_X\times \partial V_{Y,t}}((f_j-\epsilon)^+)^2\, (X,1)\cdot N_{Y,t}\d \sigma_{Y,t}\leq 0,
\end{align}
as $f_j=0$ on $\partial_\K(U_X\times V_{Y,t})$. Hence, combining \eqref{wksolcc++}-\eqref{wksolcc++b} we conclude that
\begin{align}\label{wksolcc++a1}
\iiint_{U_X\times V_{Y,t}}|\nabla_X(u-v-\epsilon)^+|^2\, \d X \d Y \d t\leq 0.
\end{align}
Finally, using, for a.e. $(Y,t)\in V_{Y,t}$, the Poincar{\'e} inequality on $U_X$ we deduce from \eqref{wksolcc++a1} that
\begin{align}\label{wksolcc++a2}
\iiint_{U_X\times V_{Y,t}}|(u-v-\epsilon)^+|^2\, \d X \d Y \d t\leq 0.
\end{align}
Hence $(u-v-\epsilon)^+=0$ a.e in $U_X\times V_{Y,t}$ and hence $u\leq v+\epsilon$ a.e. in $U_X\times V_{Y,t}$. We can conclude that we have proved the following theorem.

\begin{theorem}\label{compa} Let $u\in W(U_X\times V_{Y,t})$ be a weak sub-solution to the equation in \eqref{dpweak+llhh} in the sense of Definition \ref{weaks}. Given $u$, let
$v\in W(U_X\times V_{Y,t})$ be the unique weak solution to the problem in \eqref{dpweak+ll} in the sense of Definition \ref{weaks}. Then $u\leq v$ a.e. in $U_X\times V_{Y,t}$. Similarly, if
$u\in W(U_X\times V_{Y,t})$ is a weak super-solution to the equation in \eqref{dpweak+llhh} in the sense of Definition \ref{weaks}, then $v\leq u$ a.e. in $U_X\times V_{Y,t}$.
\end{theorem}

\section{Future research and open problems}\label{sec6}

In this paper we have initiated the study of weak solutions, and their regularity, for what we call nonlinear Kolmogorov-Fokker-Planck type equations. We believe that there are many directions to pursue in this field and in the following we formulate a number of problems.

Let $p$, $1<p<\infty$, be given and let $A=A(\xi,X,Y,t):\R^m\times\R^m\times\R^m\times\R\to\R^m$  be  continuous with respect to $\xi$, and measurable with respect to $X,Y$ and $t$. Assume that there exists a  finite constant $\Lambda\geq 1$ such that
     \begin{align}\label{assump1again}
     \Lambda^{-1}|\xi|^p\leq A(\xi,X,Y,t)\cdot\xi\leq \Lambda|\xi|^p
\end{align}
for almost every $(X,Y,t)\in\R^{N+1}$ and for all $\xi\in\mathbb{R}^m$. Given $A$ and $p$ we introduce the operator $\mathcal{L}_{A,p}$ through
\begin{eqnarray}\label{pLapKol}\mathcal{L}_{A,p}u:=\nabla_X\cdot (A(\nabla_X u(X,Y,t),X,Y,t))-(\partial_t+X\cdot\nabla_Y)u(X,Y,t).
\end{eqnarray}
This defines a class of strongly degenerate nonlinear parabolic PDEs  modelled on the classical PDE of Kolmogorov and the $p$-Laplace operator, and to our knowledge there is currently no literature devoted  to these operators. The results established in this paper concern $\mathcal{L}_{A,2}$ assuming that $A\in M(\Lambda)$ or $A\in R(\Lambda)$. We see a number of interesting research problems.

\noindent
{\bf Problem 1:} Establish existence and uniqueness of weak solutions to the Dirichlet problem
\begin{align}\label{Dppaintraa}
\begin{cases}
\mathcal{L}_{A,p}u = g^*,\quad  &\textrm{ in } U_X\times V_{Y,t},\\
u=g
,\quad  &\textrm{ on }\partial_\K(U_X\times V_{Y,t}).
\end{cases}
\end{align}

\noindent
{\bf Problem 2:} Prove higher integrability, local boundedness, Harnack inequalities and local H{\"o}lder continuity of weak solutions for the equation $\mathcal{L}_{A,p}u=0$
in the case $p\neq 2$. This is a challenging problem and the first step is probably to figure out how to replace the result of Bouchut \cite{Bou}, or the use of the fundamental solution constructed by Kolmogorov, in this case. The problem is already very interesting for the prototype
\begin{align}\label{plap}
\nabla_X\cdot(|\nabla_X u(X,Y,t)|^{p-2}\nabla_X u(X,Y,t))-(\partial_t+X\cdot\nabla_Y)u(X,Y,t)=0.
\end{align}

\noindent
{\bf Problem 3:} Consider the equation in \eqref{plap}. Prove bounds for $\nabla_Xu$ and local H{\"o}lder continuity of $\nabla_Xu$.  Note that this must be a difficult problem in the nonlinear setting due to the lack of ellipticity in the variable $Y$. Again, the right place to start is probably to (simply) consider
the equation $$\nabla_X\cdot(A(\nabla_Xu))-(\partial_t+X\cdot\nabla_Y)u=0,$$ where $A(\xi)$ has linear growth, i.e. a nonlinear $p=2$ case.

Finally, we discuss the very formulation of the Dirichlet problem. Consider the geometry of $U_X\times V_{Y,t}$ and let  $\Gamma:=\partial U_X\times V_{Y,t}$ and
\begin{align}
\Sigma^+&:=\{(X,Y,t)\in \overline{U_X}\times \partial V_{Y,t}\mid (X,1)\cdot N_{Y,t}>0\},\notag\\
\Sigma_0&:=\{(X,Y,t)\in \overline{U_X}\times \partial V_{Y,t}\mid (X,1)\cdot N_{Y,t}=0\},\notag\\
\Sigma^-&:=\{(X,Y,t)\in \overline{U_X}\times \partial V_{Y,t}\mid (X,1)\cdot N_{Y,t}<0\}.
\end{align}
Using this notation $\partial_\K(U_X\times V_{Y,t})=\Gamma\cup \Sigma^-$. Recall that $W(U_X\times V_{Y,t})$ is defined as the closure of $C^\infty(\overline{U_X\times V_{Y,t}})$ in the norm
     \begin{align}\label{weak1-+-}
     ||u||_{W(U_X\times V_{Y,t})}&:=\bigl (||u||_{L_{Y,t}^2(V_{Y,t},H_X^1(U_X))}^2+||(\partial_t+X\cdot\nabla_Y)u||_{L_{Y,t}^2(V_{Y,t},{H}_X^{-1}(U_X))}^2\bigr )^{1/2}.
    \end{align}
Assuming  $u\in W(U_X\times V_{Y,t})$,  it is relevant to define and study
the trace of  $u$ to $\Gamma\cup \Sigma^+\cup \Sigma_0\cup\Sigma^-$.  We let $=B^{2,2}_{1/2}(\partial U_X)$ denote the Besov space defined as the trace space of $H_X^1(U_X)$ to $\partial U_X$ (this space is often denoted $H^{1/2}(\partial U_X)$ in the literature). It is well known, that if $U_X$ is a bounded Lipschitz domain, then there exists
a bounded continuous non-injective operator $T:H_X^1(U_X)\to B^{2,2}_{1/2}(\partial U_X)$, called the trace operator, and a bounded continuous operator $E: B^{2,2}_{1/2}(\partial U_X)\to H_X^1(U_X)$ called the extension operator. The trace space of $L_{Y,t}^2(V_{Y,t},H_X^1(U_X))$ on $\Gamma$ is therefore $L_{Y,t}^2(V_{Y,t},B^{2,2}_{1/2}(\partial U_X))$ and
$$||u||_{L_{Y,t}^2(V_{Y,t}, B^{2,2}_{1/2}(\partial U_X))}\leq c||u||_{W(U_X\times V_{Y,t})}.$$
The trace to $\Sigma^+\cup \Sigma_0\cup\Sigma^-$ is less clear. Indeed, recall that the space $W_{X,0}(U_X\times V_{Y,t})$ is defined as the closure in the norm of
$W(U_X\times V_{Y,t})$ of $C_{X,0}^\infty(\overline{U_X\times V_{Y,t}})$. In particular, given $f\in W_{X,0}(U_X\times V_{Y,t})$ there exists $\{f_j\}$, $f_j\in C_{X,0}^\infty(\overline{U_X\times V_{Y,t}})$ such that
$$||f-f_j||_{W(U_X\times V_{Y,t})}\to 0\mbox{ as }j\to \infty,$$
and consequently,
$$||(\partial_t+X\cdot\nabla_Y)(f-f_j)||_{L_{Y,t}^2(V_{Y,t},{H}_X^{-1}(U_X))}\to 0\mbox{ as }j\to \infty.$$
Using this we see that
\begin{align}  \label{apagagain+}
&\iint_{V_{Y,t}} \langle(\partial_t+X\cdot\nabla_Y)f(\cdot,Y,t),f(\cdot,Y,t)\rangle \, \d Y\d t\notag\\
&= \lim_{j\to\infty}
\iint_{V_{Y,t}} \langle(\partial_t+X\cdot\nabla_Y)f_j(\cdot,Y,t),f_j(\cdot,Y,t)\rangle \, \d Y\d t\notag\\
&= \frac 12\lim_{j\to\infty} \iint_{\Sigma^+\cup \Sigma_0\cup\Sigma^-}f_j^2(X,Y,t)\, (X,1)\cdot N_{Y,t}\d \sigma_{Y,t}.
\end{align}
The first obstruction to a trace inequality is that $(X,1)\cdot N_{Y,t}\d \sigma_{Y,t}$ is a signed measure on $\Sigma^+\cup \Sigma_0\cup\Sigma^-$. Assuming that
$f\in W_{0}(U_X\times V_{Y,t})$ we deduce that
\begin{align}  \label{apagagain+a}
&\iint_{V_{Y,t}} \langle(\partial_t+X\cdot\nabla_Y)f(\cdot,Y,t),f(\cdot,Y,t)\rangle \, \d Y\d t\\
&= \frac 12\lim_{j\to\infty} \iint_{\Sigma^+\cup \Sigma_0}f_j^2(X,Y,t)\, (X,1)\cdot N_{Y,t}\d \sigma_{Y,t}.
\end{align}
Hence, in this case
\begin{align}  \label{apagagain+b}
\lim_{j\to\infty} \iint_{\Sigma^+\cup \Sigma_0}f_j^2(X,Y,t)\, (X,1)\cdot N_{Y,t}\d \sigma_{Y,t}\leq c||f||_{W(U_X\times V_{Y,t})},
\end{align}
and we see that we can extract a subsequence of $\{f_j\}$ converging in  $L^2(K,(X,1)\cdot N_{Y,t}\d \sigma_{Y,t})$ whenever $K$ is a compact subset of $\Sigma^+$. At the expense of additional notation the roles of $\Sigma^+$ and $\Sigma^-$ can be interchanged in this argument. This observation highlights the difficulty concerning the possibility of a trace inequality and concerning the identification of the trace space for $W(U_X\times V_{Y,t})$. This explains why we in this paper, as in \cite{LN}, have used the weaker formulation of the Dirichlet problem introduced.

\noindent
{\bf Problem 4:} What function space is the space of traces, to $\Gamma\cup \Sigma^+\cup \Sigma_0\cup\Sigma^-$, of $W(U_X\times V_{Y,t})$?

    \end{document}